\documentclass[reqno,a4paper,10pt]{amsart} 

\usepackage[utf8]{inputenc}
\usepackage[english]{babel}
\usepackage{fullpage}
\usepackage[dvipsnames]{xcolor}
\usepackage{tikz}
\usepackage{amsmath}
\usepackage{amssymb}
\usepackage{amsthm}
\usepackage{thmtools}
\usepackage[colorlinks = true, linkbordercolor = {white}, urlcolor={purple}, linkcolor={purple}, citecolor = {purple}]{hyperref}
\usepackage{cleveref}
\usepackage{amsfonts}
\usepackage{enumerate}
\usepackage[foot]{amsaddr}
\usepackage{mathtools}
\usepackage[normalem]{ulem}
\usepackage{caption}
\usepackage{pifont}
\usepackage{microtype}

\newcommand{\cmark}{\ding{51}}
\newcommand{\xmark}{\ding{55}}

\usetikzlibrary{arrows}

\theoremstyle{definition}
\newtheorem{theorem}{Theorem}[section]
\newtheorem{definition}[theorem]{Definition}
\newtheorem{corollary}[theorem]{Corollary}
\newtheorem{prop}[theorem]{Proposition}
\newtheorem{lemma}[theorem]{Lemma}
\newtheorem{example}[theorem]{Example}
\newtheorem{question}[theorem]{Question}

\theoremstyle{remark}
\newtheorem{remark}{Remark}[section]

\renewcommand{\emph}{\textsl}
\renewcommand{\textit}{\textsl}

\renewcommand{\inf}{\mathop{\mathrm{inf}\vphantom{\mathrm{sup}}}}

\newcommand{\ceil}[1]{\left\lceil #1 \right\rceil}
\newcommand{\me}{\mathrm{e}}

\newcommand{\Z}{{\mathbb Z}}

\newcommand{\R}{{\mathbb R}}
\newcommand{\N}{{\mathbb N}}

\newcommand{\mc}{\mathcal}

\newcommand{\dd}{\,\mathrm{d}}
\newcommand{\Ima}{\mathrm{Im}}

\numberwithin{equation}{section}

\makeatletter
\@namedef{subjclassname@2020}{%
  \textup{2020} Mathematics Subject Classification}
\makeatother

\begin{document}

\title{Measure theoretic entropy of random substitution subshifts}
 
\author[P.\,Gohlke]{P.\,Gohlke}
\author[A.\,Mitchell]{A.\,Mitchell}
\author[D.\,Rust]{D.\,Rust}
\author[T.\,Samuel]{T.\,Samuel}

\address[P.\,Gohlke]{Fakult\"at f\"ur Mathematik, Universit\"at Bielefeld, North Rhine-Westphalia, D-33501, Germany}
\address[A.\,Mitchell \& T.\,Samuel]{School of Mathematics, University of Birmingham, Edgbaston, B15 2TT, UK}
\address[D.\,Rust]{School of Mathematics and Statistics, The Open University, Milton Keynes, MK7 6AA, UK}



\begin{abstract}

Subshifts of deterministic substitutions are ubiquitous objects in dynamical systems and aperiodic order (the mathematical theory of quasicrystals). Two of their most striking features are that they have low complexity (zero topological entropy) and are uniquely ergodic. Random substitutions are a generalisation of deterministic substitutions where the substituted image of a letter is determined by a Markov process. In stark contrast to their deterministic counterparts, subshifts of random substitutions often have positive topological entropy, and support uncountably many ergodic measures.  The underlying Markov process singles out one of the ergodic measures, called the frequency measure. Here, we develop new techniques for computing and studying the entropy of these frequency measures. As an application of our results, we obtain closed form formulas for the entropy of frequency measures for a wide range of random substitution subshifts and show that in many cases there exists a frequency measure of maximal entropy. Further, for a class of random substitution subshifts, we prove that this measure is the unique measure of maximal entropy.  These subshifts do not satisfy Bowen's specification property or the weaker specification property of Climenhaga and Thompson and hence provide an interesting new class of intrinsically ergodic subshifts.
\end{abstract} 

\maketitle

\section{Introduction}
A (deterministic) substitution replaces each symbol in a finite or infinite string by a concatenation of symbols, according to a fixed rule. If this replacement is instead performed randomly, we speak of a \emph{random substitution}. The data necessary to determine a random substitution can be given in terms of a tuple $(\vartheta,\mathbf{P})$, where $\vartheta$ encodes all the possible replacement rules and $\mathbf{P}$ the associated probability parameters. To a given random substitution $(\vartheta,\mathbf{P})$, we associate a sequence space $X_{\vartheta}$, called a random substitution subshift. In non-degenerate cases, this subshift does not depend on the choice of $\mathbf{P}$. The bi-infinite sequences $x \in X_{\vartheta}$ are characterised by the property that every pattern in $x$ can be generated by iterating $\vartheta$, starting from a single symbol. Fundamental properties of $\vartheta$ are mirrored by topological, combinatorial and measure theoretic properties of $X_{\vartheta}$. The influence of $\mathbf{P}$ is captured by the choice of a particular probability measure $\mu_{\mathbf{P}}$ on $X_{\vartheta}$, called the \emph{frequency measure} of $(\vartheta,\mathbf{P})$. In many cases, these subshifts combine, in a non-trivial manner, properties of classic examples such as subshifts of finite type and (deterministic) substitution subshifts. In fact, these two well-studied classes can be interpreted as special cases of random substitution subshifts \cite{gohlke-rust-spindeler,rust-spindeler}.

Positive topological entropy for random substitutions was identified in the pioneering work of  Godr\`{e}che and Luck \cite{godreche-luck} in 1989, where they introduced and focused on a single example, the random Fibonacci substitution. This was later shown to hold in general for random substitutions \cite{rust-spindeler} and places them in stark contrast to their deterministic counterparts. While they have positive entropy, indicating disorder, random substitutions often admit long-range correlations presenting as a non-trivial pure-point component in the diffraction spectrum of a corresponding quasicrystal \cite{baake-spindeler-strungaru,godreche-luck,moll}. This competition between order and disorder, and between long- and short-range correlations suggests an intricate combinatorial structure which warrants careful study.

The presence of an inherent hierarchical structure allows for the application of renormalisation methods in the study of random substitutions. Leveraging these techniques, the topological entropy was calculated for several examples of random substitution subshifts, see for instance \cite{godreche-luck,nilsson}, and a unified approach was later provided in \cite{gohlke}. There, it was shown that for subshifts of \textsl{primitive} and \textsl{compatible} random substitutions, the topological entropy coincides with the notion of \textsl{inflation word entropy}, which is characterised in terms of the substitution branching process as opposed to the subshift. This builds a natural bridge to the point of view adopted in formal language theory, where random substitutions---known as (stochastic) $EOL$, or $L$ systems---are classified according to the set of accessible inflation words \cite{rozenberg,yokomori}. Similarly, the Martin boundaries of random substitutions, studied by Denker and Koslicki \cite{koslicki-denker} are limiting objects of the stochastic process induced by a random substitution, rather than being defined for the associated subshift.

Topological entropy is almost by definition blind to the generating probabilities assigned to a random substitution. This is not the case for aspects such as word frequencies and diffraction spectra, which are almost-sure properties in the limit of an appropriate substitution Markov process \cite{peyriere}. Alternatively, these properties can be associated with the frequency measure $\mu_{\mathbf{P}}$, which is ergodic with respect to the shift-action \cite{gohlke-spindeler}. It is therefore reasonable to treat entropy on the same footing, interpreting it as a quantity that is generic with respect to a frequency measure that reflects the underlying Markov process. What's more, this perspective more closely reflects the original context considered by Godr\`{e}che and Luck \cite{godreche-luck}, who were interested in random substitutions providing models for generating physical quasicrystals, whose empirical entropy will depend on the underlying Markov process.

A seminal paper of Mandelbrot on turbulence in a fluid \cite{mandelbrot}, which inspired the first formal setup of random substitutions in the physics literature \cite{peyriere}, initiated the research into fractal percolation \cite{chayes,kahane,peres}. Random substitutions have proved a useful tool to model this phenomenon \cite{dekking-meester,dekking-wal} and it was shown by Dekking, Grimmett and Meester \cite{dekking-grimmett,dekking-meester} that varying the underlying generating probabilities gives rise to several phase transitions. In the one-dimensional setting, we show that the associated entropy depends continuously on the generating probabilities and give a closed form expression in many cases. This enables us to single out those parameters that give maximal entropy.
We expect that many of the methods established in this paper can be generalized to higher dimensions, which would provide a way to determine the phase in a random percolation model that gives rise to maximal entropy.

More explicitly, we study the entropy of frequency measures corresponding to primitive random substitutions (isolated examples have been previously studied in \cite{wing}). We show that the entropy of these measures coincides with a new notion of entropy characterised in terms of inflation words (\Cref{THM:main}). For subshifts of primitive and compatible random substitutions, we demonstrate the existence of a measure of maximal entropy that is realised as a weak limit of frequency measures (\Cref{T weak limit MME}). Further, under mild conditions, we prove that there exists a frequency measure of maximal entropy, and for a large class of random substitution subshifts, we verify that this measure is the unique measure of maximal entropy (\Cref{T intrinsic ergodicity}). Indeed, determining dynamical systems which are intrinsically ergodic (i.e.\ those which exhibit a unique measure of maximal entropy) is a fundamental problem at the interface of ergodic theory and topological dynamics, and stems from the foundational work of Bowen \cite{Bowen74}.  There, it was shown  that a dynamical system which is expansive and satisfies the specification property is intrinsically ergodic. Bowen's proof relies on combinatorial arguments to establish a (weak) Gibbs property for a certain measure of maximal entropy, from which uniqueness of the measure follows. Beyond specification, for instance for $\beta$-shifts, similar strategies can be employed  \cite{THOMPSON_3,THOMPSON_1}. However, as with Bowen's proof, central to these strategies is the use of a Gibbs property. In our case there exists an obstruction to using these methods in that frequency measures of maximal entropy do not satisfy the Gibbs properties given in \cite{Bowen74,THOMPSON_3,THOMPSON_1}. Nevertheless, by establishing a weaker Gibbs property on cylinder sets of inflation words  (\Cref{LEM:mu-bound}), we are able to circumvent this obstruction  to obtain \Cref{T intrinsic ergodicity}.

\subsection*{Outline} In \Cref{S random subs} we introduce our key notation and definitions. We summarise the main results on topological entropy from \cite{gohlke} in \Cref{SS top entropy}, and give the definition of the frequency measure  corresponding to a primitive random substitution in \Cref{SS frequency measures}.

In \Cref{S MT entropy} we introduce the notion of measure theoretic inflation word entropy and state our first main result, \Cref{THM:main}, which shows, for primitive random substitutions, that this new notion of entropy coincides with the entropy of the corresponding frequency measure. 
We also obtain explicit upper and lower bounds. Under some additional assumptions, closed form expressions for the entropy can be obtained from \Cref{THM:main-urp}.

We conclude with \Cref{S MME,S examples}. \Cref{S MME} is devoted to measures of maximal entropy and intrinsic ergodicity of random substitution subshifts, and \Cref{S examples} contains a number of examples that illustrate our main results and a collection of open questions.

\section{Preliminaries}\label{S random subs}

The symbolic notation introduced in this section is mostly in line with \cite{baake-grimm,lind-marcus}, to which we refer the reader for further details. For background on random substitutions as introduced below, we point the reader to \cite{gohlke-spindeler,rust-spindeler}.

An alphabet $\mathcal{A} = \{ a_{1}, \ldots, a_{d} \}$, for some $d \in \mathbb{N}$, is a finite set of symbols $a_{i}$, which we call \textsl{letters}, equipped with the discrete topology. A \textsl{word} $u$ with letters in $\mathcal{A}$ is a finite concatenation of letters, namely $u = a_{i_{1}} \cdots a_{i_{n}}$ for some $n \in \mathbb{N}$. We write $\lvert u \rvert = n$ for the length of the word $u$, and for $m \in \mathbb{N}$, we let $\mathcal{A}^{m}$ denote the set of all words of length $m$ with letters in $\mathcal{A}$. We set $\mathcal{A}^{+} = \bigcup_{m \in \mathbb{N}} \mathcal{A}^{m}$ and let $\mathcal{A}^{\mathbb{Z}} = \{ \cdots a_{i_{-1}} a_{i_{0}} a_{i_{1}} \cdots : a_{i_j} \in \mathcal{A} \; \text{for all} \; j \in \mathbb{Z} \}$ denote the set of all bi-infinite sequences with elements in $\mathcal{A}$ and endow $\mathcal{A}^{\mathbb{Z}}$ with the product topology. With this topology, the space $\mathcal{A}^{\mathbb{Z}}$ is compact and metrisable. 

If $i$ and $j \in \mathbb{Z}$ with $i \leq j$, and $x = \cdots x_{-1} x_{0} x_{1} \cdots \in \mathcal{A}^{\mathbb{Z}}$, then we let $x_{[i,j]} = x_i x_{i+1} \cdots x_{j}$. We use the same notation if $v \in \mathcal{A}^{+}$ and $1 \leq i \leq j \leq |v|$. For $u$ and $v \in \mathcal{A}^{+}$ (or $v \in \mathcal{A}^{\Z}$), we write $u \triangleleft v$ if $u$ is a subword of $v$, namely if there exist $i$ and $j \in \mathbb{Z}$ with $i \leq j$ so that  $u = v_{[i, j]}$. 
For $u$ and $v \in \mathcal{A}^{+}$, we set $\lvert v \rvert_u$ to be the number of (possibly overlapping) occurrences of $u$ as a subword of $v$.

If $u = a_{i_1} \cdots a_{i_n}$ and $v = a_{j_1} \cdots a_{j_m} \in \mathcal{A}^{+}$, for some $n$ and $m \in \mathbb{N}$, we write $u v$ for the concatenation of $u$ and $v$, that is, we set $uv = a_{i_1} \cdots a_{i_n} a_{j_1} \cdots a_{j_m} \in \mathcal{A}^{n+m}$. The \textsl{abelianisation} of a word $u \in \mathcal{A}^{+}$ is the vector $\Phi (u) \in \N_0^d$, defined by $\Phi (u)_{i} = \lvert u \rvert_{a_i}$ for all $i \in \{ 1, \ldots, d \}$.

For a set $B$, we let $\# B$ be the cardinality of $B$ and let $\mathcal{F}(B)$ be the set of non-empty finite subsets of $B$.

\subsection{Random substitutions and their subshifts}\label{SS random subs and subshifts}

We define a random substitution via the data that is required to determine its action on letters. In the second step we extend it to a random map on words.

\begin{definition}
Let $\mathcal{A} = \{ a_{1}, \ldots, a_{d} \}$ be a finite alphabet. A random substitution $\vartheta_{\mathbf{P}} = (\vartheta, \mathbf{P})$ is a finite-set-valued function $\vartheta \colon \mathcal{A} \rightarrow \mathcal{F}(\mathcal{A}^{+})$ together with a set of non-degenerate probability vectors 
	\begin{align*}
	\mathbf{P} = \left\{ \mathbf{p}_i = ( p_{i, 1}, \ldots, p_{i, r_i} ) : r_i = \# \vartheta(a_i), \, \mathbf{p}_i \in (0,1]^{r_i} \text{ and } \sum_{j=1}^{r_i} p_{i,j} = 1 \text{ for all } 1 \leq i \leq d \right\},
	\end{align*}
such that
	\begin{align*}
	\vartheta_{\mathbf{P}} \colon a_i \mapsto
		\begin{cases}
		s^{(i,1)} & \text{with probability } p_{i, 1},\\
		\hfill \vdots \hfill & \hfill \vdots\hfill\\
		s^{(i,r_i)} & \text{with probability } p_{i, r_i},
		\end{cases}
	\end{align*}
for every $1 \leq i \leq d$, where $\vartheta(a_i) = \{ s^{(i,j)} \}_{1\leq j \leq r_i}$. We call each $s^{(i,j)}$ a \textsl{realisation} of $\vartheta_{\mathbf{P}}(a_i)$. If there exists an integer $\ell \geq 2$ such that $\lvert s^{(i,j)} \rvert = \ell$ for all $i \in \{ 1, \ldots, d \}$ and $j \in \{ 1, \ldots, r_i \}$, then we call $\vartheta_{\mathbf{P}}$ a \textsl{constant length} random substitution of length $\ell$. If $r_i = 1$ for all $i \in \{ 1, \ldots, d \}$, then we call $\vartheta_{\mathbf{P}}$ \textsl{deterministic}. 
\end{definition}

\begin{example}[Random period doubling]\label{Ex RPD}
Let $\mathcal{A} = \{ a, b\}$, and let $p \in (0,1)$.
The \textsl{random period doubling substitution} $\vartheta_{\mathbf{P}} = (\vartheta, \mathbf{P})$ is the constant length substitution given by
	\begin{align*}
	\vartheta_{\mathbf P} \colon
		\begin{cases}
		a \mapsto
			\begin{cases}
			ab & \text{with probability } p,\\
			ba & \text{with probability } 1-p,
		     \end{cases}\\[1.25em]
		b \mapsto aa \quad \text{with probability } 1,
		\end{cases}
	\end{align*}
with defining data $r_{a} = 2$, $r_{b} = 1$, $s^{(a, 1)} = ab$, $s^{(a, 2)} = ba$, $s^{(b, 1)} = aa$, $\mathbf{P} = \{ \mathbf{p}_{a} = (p, 1-p), \mathbf{p}_{b} = (1) \}$, and corresponding set-valued function $\vartheta \colon a \mapsto \{ab,ba\}, b \mapsto \{aa\}$.
\end{example}

In the following we describe how a random substitution $\vartheta_{\mathbf{P}}$ determines a (countable state) Markov matrix $Q$, indexed by $\mathcal{A}^{+} \times \mathcal{A}^{+}$. We interpret the entry $Q_{u,v}$ as the probability to map a word $u$ to a word $v$ under the random substitution. Formally, $Q_{a_i, s^{(i,j)}} = p_{i,j}$ for $j \in \{1,\ldots, r_i\}$ and $Q_{a_i,v} =0$ if $v \notin \vartheta(a_i)$.  We extend the action of $\vartheta_{\mathbf{P}}$ to finite words by mapping each letter \textsl{independently} to one of its realisations, distinguishing random substitutions from $S$-adic systems. More precisely, given $n \in \mathbb{N}$, $u = a_{i_1} \cdots a_{i_n} \in \mathcal{A}^{n}$ and $v \in \mathcal{A}^{+}$ with $|v| \geq n$, we let
	\begin{align*}
	\mathcal{D}_n(v) = \{ (v^{(1)},\ldots, v^{(n)}) \in (\mathcal{A}^{+})^{n} : v^{(1)} \cdots v^{(n)} = v \}
	\end{align*} 
denote the set of all decompositions of $v$ into $n$ individual words and set
	\begin{align*}
	Q_{u,v} = \sum_{(v^{(1)},\ldots,v^{(n)}) \in \mathcal{D}_n(v)} \prod_{j = 1}^{n} Q_{a_{i_j},v^{(j)}}.
	\end{align*} 
In words, $\vartheta_{\mathbf{P}}(u) = v$ with probability $Q_{u,v}$.

For $u \in \mathcal{A}^{+}$, let $(\vartheta_{\mathbf{P}}^{n}(u))_{n \in \mathbb{N}}$ be a stationary Markov chain on some probability space $(\Omega_u, \mathcal{F}_u, \mathbb{P}_u)$, with Markov matrix given by $Q$, that is
	\begin{align*}
	\mathbb{P}_u [\vartheta_{\mathbf{P}}^{n+1}(u) = w \mid \vartheta_{\mathbf{P}}^{n}(u) = v] = \mathbb{P}_v [\vartheta_{\mathbf{P}}(v) = w] = Q_{v,w},
	\end{align*} 
for all $v$ and $w \in \mathcal{A}^{+}$, and $n \in \mathbb{N}$. In particular, we have
	\begin{align*}
	\mathbb{P}_u [\vartheta_{\mathbf{P}}^{n}(u) = v] = (Q^{n})_{u,v}
	\end{align*} 
for all $u$ and $v \in \mathcal{A}^{+}$, and $n \in \mathbb{N}$. We often write $\mathbb{P}$ for $\mathbb{P}_u$ if the initial word is understood. 
In this case, we also write $\mathbb{E}$ for the expectation with respect to $\mathbb{P}$.
As before, we call $v$ a \textsl{realisation} of $\vartheta^{n}_{\mathbf{P}}(u)$ if $(Q^{n})_{u,v} > 0$ and set 
	\begin{align*}
	\vartheta^{n}(u) = \{ v \in \mathcal{A}^{+} : (Q^{n})_{u,v} > 0\}
	\end{align*} 
to be the set of all realisations of $\vartheta_{\mathbf{P}}^{n}(u)$. Conversely, we may regard $\vartheta^{n}_{\mathbf{P}}(u)$ as the set $\vartheta^{n}(u)$, endowed with the additional structure of a probability vector. If $u = a \in \mathcal{A}$ is a letter, we call a word $v \in \vartheta^{k}(a)$ a (level-$k$) \textsl{inflation word}.
The approach of defining a random substitution in terms of an associated Markov chain goes back to work of Peyri\`{e}re \cite{peyriere} and was pursued further by Koslicki \cite{koslicki}, and Denker and Koslicki \cite{koslicki-denker}.

For many structural properties of $\vartheta_{\mathbf{P}}$ the choice of (non-degenerate) probability vectors is immaterial. In these cases, one sometimes refers to $\vartheta$ instead of $\vartheta_{\mathbf{P}}$ as a random substitution, see for instance \cite{gohlke}. On the other hand, for some applications, one needs additional structure on the probability space. In fact, there is an underlying branching process, similar to a Galton--Watson process, that allows one to construct more refined random variables, see \cite{gohlke-spindeler} for further details.

Given a random substitution $\vartheta_{\mathbf{P}} = (\vartheta, \mathbf{P})$ over an alphabet $\mathcal{A} = \{ a_{1}, \ldots, a_{d} \}$ with cardinality $d \in \mathbb{N}$, we define the \textsl{substitution matrix} $M = M_{\vartheta_{\mathbf{P}}} \in \mathbb{R}^{d \times d}$ of $\vartheta_{\mathbf{P}}$ by
	\begin{align*}
	M_{i, j}
	= \mathbb{E}[\lvert \vartheta_{\mathbf{P}}  (a_{j}) \rvert_{a_{i}}]
	= \sum_{k = 1}^{r_{j}} p_{j, k} \lvert s^{(j, k)} \rvert_{a_{i}}.
	\end{align*}
Since $M$ has only non-negative entries, its spectral radius is also a real eigenvalue of maximal modulus, denoted by $\lambda$. For notational convenience, we denote the maximal length of a (level-$1$) inflation word by
\[
|\vartheta| = \max \{ |u| : u \in \vartheta(a), a \in \mathcal A \}.
\]
By construction, $1 \leq \lambda \leq |\vartheta|$, where $\lambda = 1$ occurs precisely if $M$ is column-stochastic. This corresponds to the trivial case of a non-expanding random substitution, which we discard in the following. If the matrix $M$ is \emph{primitive} (i.e. if there exists a $k \in \mathbb{N}$ such that all the entries of $M^{k}$ are positive), Perron--Frobenius theory implies that $\lambda$ is a simple eigenvalue and that the corresponding left and right eigenvectors $\mathbf{L} = (L_{1}, \ldots, L_{d})^{\top}$ and $\mathbf{R} = (R_{1}, \ldots, R_{d})^{\top}$ can be chosen to have strictly positive entries. We normalise these eigenvectors according to $\lVert \mathbf{R} \rVert_{1} = 1 = \mathbf{L}^{\top} \, \mathbf{R}$. In this situation, we call $\lambda$ the \textsl{Perron--Frobenius eigenvalue} of $\vartheta_{\mathbf{P}}$, and $\mathbf{L}$ and $\mathbf{R}$ the \textsl{left} and \textsl{right Perron--Frobenius eigenvectors} of $\vartheta_{\mathbf{P}}$, respectively.

\begin{definition}
We say that $\vartheta_{\mathbf{P}}$ is \textsl{primitive} if $M = M_{\vartheta_{\mathbf{P}}}$ is primitive and its Perron--Frobenius eigenvalue satisfies $\lambda > 1$.
\end{definition}

We emphasise that for a random substitution $\vartheta_{\mathbf{P}}$, being primitive is independent of the (non-degenerate) data $\mathbf{P}$. In this sense, primitivity is a property of $\vartheta$ rather than $\vartheta_{\mathbf{P}}$.

\begin{remark}
Primitivity is a standard assumption, both for deterministic and random substitutions. More general (random) substitutions can be treated by bringing $M$ into an upper block-triangular normal form via an appropriate permutation of letters. Throughout most of this paper we stick to the primitive case to avoid technicalities.
\end{remark}

For a constant length primitive random substitution of length $\ell$, an elementary calculation shows that $\lambda = \ell$; and for a given primitive random substitution $\vartheta_{\mathbf{P}}$ with Perron--Frobenius eigenvalue $\lambda$ and for $k \in \mathbb{N}$, we have that $M_{\vartheta_{\mathbf{P}}}^{k} = M_{\vartheta_{\mathbf{P}}^{k}}$ and hence the Perron--Frobenius eigenvalue of $\vartheta_{\mathbf{P}}^{k}$ is $\lambda^{k}$.

Given a random substitution $\vartheta_{\mathbf{P}} = (\vartheta, \mathbf{P})$, a word $u \in \mathcal{A}^{+}$ is called \textsl{($\vartheta$-)legal} if there exists an $a_i \in \mathcal{A}$ and $k \in \mathbb{N}$ such that $u$ appears as a subword of some word in $\vartheta^{k} (a_i)$. We define the \textsl{language} of $\vartheta$ by $\mathcal{L}_{\vartheta} = \{ u \in \mathcal{A}^{+} : u \text{ is $\vartheta$-legal} \}$ and, for $w \in \mathcal{A}^{+} \cup \mathcal{A}^{\mathbb{Z}}$, we let $\mathcal{L} (w) = \{ u \in \mathcal{A}^{+} : u \triangleleft w \}$ denote the language of $w$.

\begin{definition}
The \textsl{random substitution subshift} of a random substitution $\vartheta_{\mathbf{P}} = (\vartheta, \mathbf{P})$ is the system $(X_{\vartheta}, S)$, where $X_{\vartheta} = \{ w \in \mathcal{A}^{\mathbb{Z}} : \mathcal{L} (w) \subseteq \mathcal{L}_{\vartheta} \}$ and $S$ denotes the (left) shift map, defined by $S (w)_{i} = w_{i+1}$ for each $w \in X_{\vartheta}$. 
\end{definition}

If $\vartheta_{\mathbf{P}}$ is primitive, the corresponding sequence space $X_{\vartheta}$ is always non-empty \cite{gohlke-spindeler}.
The notation $X_{\vartheta}$ mirrors the fact that the random substitution subshift does not depend on the choice of $\mathbf{P}$.  We endow $X_{\vartheta}$ with the subspace topology inherited from $\mathcal{A}^{\mathbb{Z}}$, and since $X_{\vartheta}$ is defined in terms of a language, it is a compact $S$-invariant subspace of $\mathcal{A}^{\mathbb{Z}}$. Hence, $X_{\vartheta}$ is a subshift. For $n \in \mathbb{N}$, we write $\mathcal{L}_{\vartheta}^{n} = \mathcal{L}_\vartheta \cap \mathcal{A}^{n}$ and $\mathcal{L}^{n} (w) = \mathcal{L}(w) \cap \mathcal{A}^{n}$ to denote the subsets of $\mathcal{L}_{\vartheta}$ and $\mathcal{L} (w)$, respectively, consisting of words of length $n$.  We also note that, when $\vartheta$ is primitive, $X_{\vartheta^{k}} = X_{\vartheta}$ for all $k \in \mathbb{N}$.

The set-valued function $\vartheta$ naturally extends to $X_{\vartheta}$, where for $w = \cdots w_{-1} w_{0} w_{1} \cdots \in X_{\vartheta}$ we let $\vartheta(w)$ denotes the (infinite) set of sequences of the form $v = \cdots v_{-2} v_{-1}.v_0 v_1 \cdots$, with $v_j \in \vartheta(w_j)$ for all $j \in \mathbb{Z}$. By definition, it is easily verified that $\vartheta(X_{\vartheta}) \subset X_{\vartheta}$. Some properties of $\vartheta$ are reminiscent of continuous functions, although $\vartheta$ itself is \emph{not} a function. The following property will be useful in our discussion of intrinsic ergodicity (\Cref{SS intrinsic ergodicity}) and is also of independent interest.

\begin{lemma}\label{L compact to compact}
If $\vartheta_{\mathbf{P}} = (\vartheta, \mathbf{P})$ is a random substitution and $X \subset \mathcal{A}^{\mathbb{Z}}$ is compact, then $\vartheta(X)$ is compact.
\end{lemma} 

\begin{proof}
It suffices to show that $\vartheta(X)$ is closed. Let $(y^{(n)})_{n \in {\mathbb{N}}}$ denote a sequence in $\vartheta(X)$ and assume that this sequence converges to some $y \in \mathcal{A}^{\mathbb{Z}}$. We need to show that $y \in \vartheta(X)$. To this end, let $(x^{(n)})_{n \in \mathbb{N}}$ be a sequence in $X$ with $y^{(n)} \in \vartheta(x^{(n)})$ for all $n \in \mathbb{N}$. By compactness of $X$, this sequence has an accumulation point $x = \cdots x_{-1} x_{0} x_{1} \cdots \in X$. By restricting to an appropriate subsequence, we may assume that 
	\begin{align*}
	x^{(m)}_{[-n,n]} = x^{\phantom{(}}_{[-n,n]}
	\end{align*}
for all $m$ and $n \in {\mathbb{N}}$ with $m \geq n$.  In which case,
	\begin{align*}
	y^{(n)}_{\,[-n, n]} = w^{(n)}_{-n\phantom{]}} \cdots w^{(n)}_{-1\phantom{]}} . w^{(n)}_{0\phantom{]}} \cdots w^{(n)}_{n\phantom{]}} 
	\end{align*}
with $w^{(n)}_j \in \vartheta(x_j)$ for all $j \in \{-n, \ldots, n\}$. As $( y^{(m)} )_{m \in \mathbb{N}}$ converges to $y$, we may assume, for $n \in \mathbb{N}$,
	\begin{align*}
	y_{\,[-n, n]}  = w^{(n)}_{-n} \cdots w^{(n)}_{-1}. w^{(n)}_{0} \cdots w^{(n)}_n,
	\end{align*}
again by possibly restricting to an appropriate subsequence. By a standard diagonal argument utilising the pigeonhole principle, we can choose $w_j \in \vartheta(x_j)$ for all $j \in {\mathbb{Z}}$ such that $y = \cdots w_{-2} w_{-1}. w_0 w_1 w_2 \cdots$. Namely, we have that $y \in \vartheta(x)$.
\end{proof}

\subsection{Special classes of random substitutions}

Primitive random substitutions produce a wide variety of subshifts, including for example all topologically transitive shifts of finite type \cite{gohlke-rust-spindeler} as well as all (primitive) deterministic substitution subshifts. It is therefore reasonable to expect that further assumptions on the random substitution are required in order to obtain a more detailed control over its (measure theoretic) entropy. Indeed, there is a useful property which allows us to obtain more precise estimates that can be shown to fail in the general primitive setting. 
Recall that for $v = v_1 \cdots v_n$ the random word $\vartheta_{\mathbf{P}}(v) = \vartheta_{\mathbf{P}}(v_1) \cdots \vartheta_{\mathbf{P}}(v_n)$ can be written as a concatenation of the random variables $\vartheta_{\mathbf{P}}(v_1),\ldots,\vartheta_{\mathbf{P}}(v_n)$. In general, there might be several realisations of $(\vartheta_{\mathbf{P}}(v_1),\ldots,\vartheta_{\mathbf{P}}(v_n))$ that concatenate to the same realisation of $\vartheta_{\mathbf{P}}(v)$. In some situations this phenomenon can be excluded.

\begin{definition}
We say that $\vartheta_{\mathbf{P}}$ has \emph{unique realisation paths} if for every $v \in \mathcal L_{\vartheta}^n$ and $k \in \N$, the random variable $(\vartheta_{\mathbf{P}}^k(v_1), \ldots, \vartheta_{\mathbf{P}}^k(v_n))$ is completely determined by $\vartheta_{\mathbf{P}}^k(v)$.
\end{definition}

While the definition above is most adequate for our purposes, it is worth pointing out that the property of having unique realisation paths does not depend on the choice of $\mathbf{P}$. Indeed, it is straightforward to verify that $\vartheta_{\mathbf{P}}$ has unique realisation paths if and only if for all $v \in \mc L_{\vartheta}^n$ and $k \in \N$ the concatenation map 
\[
\vartheta^k(v_1) \times \cdots \times \vartheta^k(v_n) \to \mc L_{\vartheta}, \quad
(w_1,\ldots,w_n) \mapsto w_1 \cdots w_n
\]
is injective.

The property of having unique realisation paths might appear difficult to check in general. However, there is a general class of random substitutions that satisfy this condition and that is of relevance in the context of random tilings. In the following, we denote by a \emph{marginal} of $\vartheta_{\mathbf{P}}$ a deterministic substitution $\varrho$ on the same alphabet $\mc A$, such that $\varrho(a) \in \vartheta(a)$ for all $a \in \mc A$.

\begin{definition}
We say that a primitive random substitution $\vartheta_{\mathbf{P}}$ is \emph{geometrically compatible} if there is a real number $\lambda > 1$ and a vector $\mathbf{L}$ with strictly positive entries, such that $\mathbf{L}$ is a left eigenvector with eigenvalue $\lambda$ for all marginals of $\vartheta_{\mathbf{P}}$.
\end{definition}

In this situation, it is easy to check that $\lambda$ and $\mathbf{L}$ are indeed Perron--Frobenius data for the substitution matrix $M$ of $\vartheta_{\mathbf{P}}$.
Geometric compatibility is equivalent to the assumption that $\lambda$ and the corresponding eigenline spanned by $\mathbf{L}$ are independent of the choice of $\mathbf{P}$, which is easy to check for a given example. Moreover, it provides a natural setting in which a random substitution can be interpreted as a random inflation rule on an associated tiling dynamical system. In this geometric model, every letter $a_i$ is identified with a tile of length $\mathbf{L}_i$. This motivates the term \emph{geometrically compatible}. 

\begin{remark}
The class of geometrically compatible random substitutions contains all (primitive) constant-length random substitutions. Indeed, if $\vartheta_{\mathbf{P}}$ is of length $\ell$ we have $\lambda = \ell$ and $\mathbf{L}_i = 1$ for all $1 \leq i \leq d$, irrespective of $\mathbf{P}$.
\end{remark}

Geometric compatibility is also a generalization of primitive \emph{compatible} random substitutions. Compatibility has been a standard assumption in much recent work on random substitutions and is particularly useful in those settings, where $\vartheta$ instead of $\vartheta_{\mathbf{P}}$ is regarded as a random substitution.

\begin{definition}
We say that a random substitution $\vartheta_{\mathbf{P}} = (\vartheta, \mathbf{P})$ is \textsl{compatible} if for all $a \in \mathcal{A}$, and $u$ and $v \in \vartheta(a)$, we have $\Phi (u) = \Phi (v)$. 
\end{definition}

Observe that compatibility is independent of the choice of probabilities, and that a random substitution $\vartheta_{\mathbf{P}} = (\vartheta, \mathbf{P})$ is compatible if and only if for all $u \in \mathcal{A}^{+}$, we have that $\lvert s \rvert_{a} = \lvert t \rvert_{a}$ for all $s$ and $t \in \vartheta (u)$, and $a \in \mathcal{A}$. We write $\lvert \vartheta (u) \rvert_{a}$ to denote this common value, and let $\lvert \vartheta (u) \rvert$ denote the common length of words in $\vartheta (u)$. In which case, letting $M = M_{\vartheta_{\mathbf{P}}}$ denote the substitution matrix of $\vartheta_{\mathbf{P}}$, we have that $M_{i, j} = \lvert \vartheta (a_j) \lvert_{a_i}$ for all $a_i$ and $a_j \in \mathcal{A}$. Note that the random period doubling substitution defined in \Cref{Ex RPD} is compatible, since $\Phi (ab) = \Phi (ba) = (1,1)^{\top}$, and is primitive, since the square of its substitution matrix is positive.

The class of geometrically compatible random substitutions contains all compatible random substitutions and all constant length random substitutions but is not confined to them.

\begin{example}
Let $\vartheta_{\mathbf{P}}$ be the primitive random substitution on the alphabet $\mathcal A = \{a,b\}$ defined by
\[
\vartheta_{\mathbf{P}} \colon \begin{cases}
a \mapsto & abb,
\\ b \mapsto & \begin{cases}
a & \mbox{with probability } p,
\\bb & \mbox{with probability } 1-p.
\end{cases}
\end{cases}
\]
This random substitution is geometrically compatible with $\mathbf{L} = (2,1)^{\top}$ and $\lambda = 2$. It is neither of constant length nor compatible. 
\end{example}

\begin{example}
Let $\vartheta_{\mathbf{P}}$ be the primitive random substitution defined by
\[
\vartheta_{\mathbf{P}} \colon a \mapsto \begin{cases}
a & \mbox{with probability } p,
\\ab & \mbox{with probability } 1-p,
\end{cases}
\qquad 
b \mapsto \begin{cases}
a & \mbox{with probability } q,
\\ba & \mbox{with probability } 1-q.
\end{cases}
\]
This is neither geometrically compatible nor does it have unique realisation paths. The latter can be seen from the fact that both $(a,ba)$ and $(ab,a)$ are two different realisations of $(\vartheta_{\mathbf{P}}(a), \vartheta_{\mathbf{P}}(b))$ that give rise to the same word $aba \in \vartheta(ab)$.
\end{example}

\begin{remark}
Like primitivity, geometric compatibility is stable under taking powers of the random substitution at hand. That is, if $\vartheta_{\mathbf{P}}$ is geometrically compatible, then so is $\vartheta_{\mathbf{P}}^n$ for all $n \in \N$. This is because the Perron--Frobenius data $(\lambda,\mathbf{L})$ of $\vartheta_{\mathbf{P}}$ is independent of $\mathbf{P}$, which in turn implies that the Perron--Frobenius data $(\lambda^n,\mathbf{L})$ of $\vartheta_{\mathbf{P}}^{n}$ is independent of $\mathbf{P}$.
\end{remark}

\begin{lemma}
Every primitive, geometrically compatible random substitution has unique realisation paths.
\end{lemma}

\begin{proof}
Let $\vartheta_{\mathbf{P}}$ be primitive and geometrically compatible. Since the same holds for $\vartheta_{\mathbf{P}}^k$, we may restrict to the case $k = 1$ in the following.
Let $v \in \mathcal L_{\vartheta}^{n}$ and let $u$ be a realisation of the random word
\[
\vartheta_{\mathbf{P}}(v) = \vartheta_{\mathbf{P}}(v_1) \cdots \vartheta_{\mathbf{P}}(v_n)
\]
and $(u^1, \ldots, u^n)$ a corresponding realisation of $(\vartheta_{\mathbf{P}}(v_1),\ldots, \vartheta_{\mathbf{P}}(v_n))$ satisfying 
\[
u = u^1 \cdots u^n.
\]
Let $M_1$ be the substitution matrix of a marginal of $\vartheta_{\mathbf{P}}$ with $v_1 \mapsto u_1$.
Since $\mathbf{L}$ has strictly positive entries, there is a unique $1\leq m \leq |u|$ such that
\[
\mathbf{L} \Phi(u_{[1,m]}) = \mathbf{L} \Phi(u^1) = \mathbf{L} M_1 \Phi(v_1) = \lambda \mathbf{L}_{v_1}.
\]
This determines $u^1 = u_{[1,m]}$ unambiguously. Inductively, we find that $u^k$ is uniquely determined by $u$ for all $1\leq k \leq n$. 
\end{proof}

For the reader's convenience, we summarize the relation between different characterisations of primitive random substitutions in Figure~\ref{FIG:1}.

\begin{figure}

\centering

\begin{tikzpicture} 

\node (1) at (-2.000000,0.000000) {Constant Length};

\node (2) at (2.000000,0.000000) {Compatibility}; 

\node (3) at (0.000000,-1.50000) {Geometric Compatibility}; 

\node (4) at (0.000000,-3.00000) {Unique Realisation Paths};

\draw [shorten >=0.5em,shorten <=0em,-implies, double equal sign distance] (1) -- (3);

\draw [shorten >=0.5em,shorten <=0em,-implies, double equal sign distance] (2) -- (3);

\draw [shorten >=0.3em,shorten <=0.3em,-implies, double equal sign distance] (3) -- (4);

\end{tikzpicture}

\caption{Implication diagram for some conditions on primitive random substitutions.} \label{FIG:1}

\end{figure}
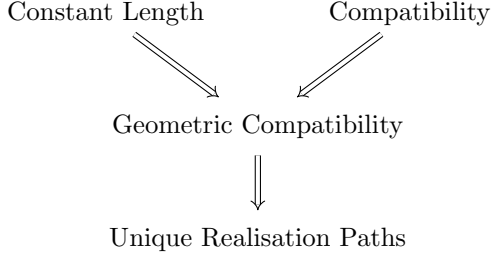

\subsection{Topological entropy}\label{SS top entropy}

The non-trivial topological entropy of random substitution subshifts distinguishes them from subshifts of deterministic substitutions, which always have zero topological entropy, see \cite{queffelec}. The topological entropy was calculated for several families of random substitutions in \cite{godreche-luck,nilsson}. There, the topological entropy was calculated from the growth rate of inflation words. This approach was unified by Gohlke \cite{gohlke}, where the notion of \textsl{inflation word entropy} was introduced for compatible primitive random substitutions and shown to equal the topological entropy of the corresponding subshift.

For completeness, let us take a moment to recall the definition of the topological entropy of a subshift, see \cite{MR1963683,MR648108} for further details. Given a subshift $(X,S)$, we define the language of the subshift by $\mc L(X) = \{x_{[j,k]} : x \in X, j\leq k \}$ and for each $n \in \mathbb{N}$, we let $\mathcal{L}^{n} (X)$ denote the set of all words of length $n$ in $\mathcal{L} (X)$. The \textsl{topological entropy} $h_{\text{top}}(X)$ of the system $(X, S)$ is defined to be the quantity
	\begin{align}\label{eq:top_entropy_constant}
	h_{\text{top}}(X) = \lim_{n \to \infty} \frac{1}{n} \log( \# \mathcal{L}^{n} (X)).
	\end{align}
Given a primitive and compatible random substitution $\vartheta_{\mathbf{P}} = (\vartheta, \mathbf{P})$ over the alphabet $\mathcal{A} = \{ a_1, \ldots, a_d \}$, we have that $\mathcal{L} (X_{\vartheta}) = \mathcal{L}_{\vartheta}$. For each $m \in \mathbb{N}$, let $\mathbf{q}_{m} = (q_{m, 1}, \dots, q_{m, d})$ denote the vector defined by
	\begin{align}\label{eq:refece_page_17}
	q_{m,i} = \log (\# \vartheta^{m} (a_i))
	\end{align}
for $i \in \{ 1, \ldots, d \}$. When the limit exists, the inflation word entropy of type $i$ is defined by
	\begin{align*}
	t_{i} (\vartheta_{\mathbf{P}}) = t_{i} (\vartheta) = \lim_{m \rightarrow \infty} \frac{q_{m,i}}{\lvert \vartheta^{m} (a_i) \rvert}.
	\end{align*}

\begin{theorem}[{\cite[Theorem~17]{gohlke}}]\label{T topological inflation entropy}
Let $\vartheta_{\mathbf{P}} = (\vartheta, \mathbf{P})$ be a primitive and compatible random substitution over the alphabet $\mathcal{A} = \{ a_1, \ldots, a_d \}$ with cardinality $d \in \mathbb{N}$. Let $\lambda$ denote the Perron--Frobenius eigenvalue of $\vartheta_{\mathbf{P}}$, and let $\mathbf{R}$ be the right Perron--Frobenius eigenvector of $\vartheta_{\mathbf{P}}$. For all $i \in \{ 1, \ldots, d \}$, the inflation word entropy $t_i (\vartheta)$ exists, is independent of $i$, and is equal to the topological entropy $h_{\text{top}}(X_{\vartheta})$ of the system $(X_{\vartheta}, S)$. Moreover, for all $m \in \mathbb{N}$, we have
	\begin{align}\label{E top entropy bounds}
	\frac{1}{\lambda^{m}} \mathbf{q}_{m}^{\top} \mathbf{R} \leq t_{i} (\vartheta) 
	= h_{\text{top}} (X_{\vartheta}) \leq \frac{1}{\lambda^{m}-1} \mathbf{q}_{m}^{\top} \mathbf{R},
	\end{align}
where the lower bounds are non-decreasing in $m$. Further, $h_{\text{top}}(X_{\vartheta})$ can be calculated as
	\begin{align*}
	h_{\text{top}} (X_{\vartheta})
	= t_{i} (\vartheta)
	= \lim_{m \rightarrow \infty} \frac{1}{\lambda^{m}} \mathbf{q}_{m}^{\top} \mathbf{R}
	= \sup_{m \in \mathbb{N}} \frac{1}{\lambda^{m}} \mathbf{q}_{m}^{\top} \mathbf{R}.
	\end{align*}
\end{theorem}

In general, it is difficult to obtain a closed form formula for the topological entropy using \Cref{T topological inflation entropy}. The difficulty lies in quantifying the overlaps of sets of the form $\vartheta^{m} (u)$, for $u \in \vartheta(a_i)$. However, if the random substitution satisfies either of two mild conditions, then it is possible to obtain a closed form expression for the topological entropy using \Cref{T topological inflation entropy}.

\begin{definition}
A random substitution $\vartheta_{\mathbf{P}} = (\vartheta, \mathbf{P})$ is said to satisfy the \textsl{identical set condition} if
	\begin{align*}
	u \; \text{and} \; v \in \vartheta (a) &\implies \vartheta^{k} (u) = \vartheta^{k} (v)
	\end{align*}
for all $a \in \mathcal{A}$ and $k \in \mathbb{N}$. It is said to satisfy the \textsl{disjoint set condition} if
	\begin{align*}
	u  \; \text{and} \;  v \in \vartheta (a) \; \text{with} \; u \neq v &\implies \vartheta^{k} (u) \cap \vartheta^{k} (v) = \varnothing
	\end{align*}
for all $a \in \mathcal{A}$ and $k \in \mathbb{N}$.
\end{definition}

\begin{remark}
An easy way to satisfy the identical set condition is to assume that $\vartheta(a) = \vartheta(b)$ for all $a, b \in \mc A$. In this case, the corresponding random substitution subshift is a coded shift, generated by the set $\vartheta(a)$. However, this structure is not \emph{necessary} for the identical set condition as one may see from the example $\vartheta \colon a,b \mapsto \{abc,bac\}, c \mapsto \{a\}$. For further discussion of the identical set condition and the disjoint set condition we refer to the examples in \Cref{S examples} and \cite{gohlke}.
\end{remark}

\begin{corollary}[{\cite[Corollary~18]{gohlke}}]\label{C IS DS topological entropy}
Assume the setting of \Cref{T topological inflation entropy}. If $\vartheta_{\mathbf{P}}$ satisfies the identical set condition, then
	\begin{align*}
	h_{\text{top}} (X_{\vartheta}) = \frac{1}{\lambda} \mathbf{q}_{1}^{\top} \mathbf{R} = \frac{1}{\lambda} \sum_{i=1}^{d} R_{i} \log(\# \vartheta (a_i)).
	\end{align*}
If $\vartheta_{\mathbf{P}}$ satisfies the disjoint set condition, then
	\begin{align*}
	h_{\text{top}} (X_{\vartheta}) = \frac{1}{\lambda-1} \mathbf{q}_{1}^{\top} \mathbf{R} = \frac{1}{\lambda-1} \sum_{i=1}^{d} R_{i} \log(\# \vartheta (a_i)).
	\end{align*}
\end{corollary}

Thus, if $\vartheta_{\mathbf{P}}$ satisfies the identical set condition, then the topological entropy of its subshift achieves the lower bound given in \eqref{E top entropy bounds} with $m=1$, and if $\vartheta_{\mathbf{P}}$ satisfies the disjoint set condition, then it achieves the upper bound given in \eqref{E top entropy bounds} with $m=1$.  In fact, one can show that these bounds are attained precisely when $\vartheta_{\mathbf{P}}$ satisfies the identical/disjoint set condition. The random period doubling substitution defined in \Cref{Ex RPD} satisfies the disjoint set condition. Hence, it follows by \Cref{C IS DS topological entropy} that the corresponding subshift has topological entropy equal to $\log(2^{2/3})$, noting that $\lambda = 2$ and $\mathbf{R} = (\frac{2}{3},\frac{1}{3})^{\top}$.

\subsection{Frequency measures}\label{SS frequency measures}

For $v \in \mathcal{L} (X)$ and $m \in \mathbb{Z}$, we define the cylinder set of $v$ at position $m$ by
	\begin{align*}
	[v]_{m} = \{ w \in X : w_{m + i} = v_{i} \text{ for all } 0 \leq i \leq \lvert v \rvert - 1 \}
	\end{align*}
and set $[v] = [v]^{}_0$ for convenience.
The union of the collection of cylinder sets that specify the zeroth position,
	\begin{align*}
	\xi (X) = \{ [v]_{m} : v \in \mathcal{L}_{\vartheta}, \, 1 - \lvert v \rvert \leq m \leq 0 \} \},
	\end{align*}
with $\{ \varnothing\}$ forms a semi-ring of sets, which generates the Borel $\sigma$-algebra $\mathcal{B} (X)$.  Hence, any content with mass one defined on $\xi (X) \cup \{ \varnothing \}$ extends uniquely to a probability measure on $\mathcal{B} (X)$ by the Hahn-Kolmogorov extension theorem. As we will see shortly, frequency measures are defined in this manner. 

Given a primitive random substitution $\vartheta_{\mathbf{P}}$, the \textsl{expected frequency} of a word $v \in \mathcal{L}_{\vartheta}$ is defined by
	\begin{align*}
	\text{freq} (v)
	= \lim_{k \rightarrow \infty} \frac{\mathbb{E} [ \lvert \vartheta_{\mathbf{P}}^{k} (a) \rvert_{v} ] }{ \mathbb{E} [\lvert \vartheta_{\mathbf{P}}^{k} (a)\rvert ]},
	\end{align*}
where this limit is independent of the choice of $a \in \mathcal{A}$. In fact, we have the stronger property that the word frequencies exist $\mathbb{P}$-almost surely in the limit of large inflation words and are given by $\text{freq}(v)$ for all $v \in \mathcal{L}_{\vartheta}$, see \cite{gohlke-spindeler} for further details. It turns out that these frequencies naturally define an ergodic measure supported on $X_{\vartheta}$.

\begin{prop}[{\cite[Proposition~5.3, Theorem~5.9]{gohlke-spindeler}}]\label{P frequency measure}
Let $\vartheta_{\mathbf{P}}$ be a primitive random substitution with subshift $X_{\vartheta} \neq \varnothing$. Define $\mu_{\mathbf{P}} \colon \xi (X_{\vartheta}) \cup \{ \varnothing \} \rightarrow [0,1]$ by $\mu_{\mathbf{P}} (\varnothing) = 0$, $\mu (X_{\vartheta}) = 1$, and $\mu_{\mathbf{P}} ([v]_{m}) = \text{freq} (v)$ for $v \in \mathcal{L}_{\vartheta}$ and $m \in \{ 1 - \lvert v \rvert, 2 - \lvert v \rvert, \ldots, 0 \}$.  The set function $\mu_{\mathbf{P}}$ is a content with mass one which extends uniquely to a shift-invariant ergodic probability measure on $\mathcal{B} (X_{\vartheta})$.
\end{prop}

We call the measure $\mu_{\mathbf{P}}$ defined in \Cref{P frequency measure} the \textsl{frequency measure} corresponding to the random substitution $\vartheta_{\mathbf{P}}$. Alternatively, frequency measures can be defined in terms of the right Perron--Frobenius eigenvector of a sequence of induced random substitutions (treating words as letters), which encode information on word frequencies; in particular,
	\begin{align}\label{eq:freq_PF_right}
	\mu_{\mathbf{P}}([a]) = \text{freq} (a) = R_{a}
	\quad \text{and} \quad 
	\lim_{k \to \infty} \frac{\mathbb{E}[\lvert \vartheta_{\mathbf{P}}^{k} (a) \rvert]}{\mathbb{E}[\lvert \vartheta_{\mathbf{P}}^{k-1} (a) \rvert]} = \lambda,
	\end{align}
for $a \in \mathcal{A}$ -- see \cite{gohlke-spindeler} for further details.

Observe that frequency measures are dependent on the probabilities of the substitution. As such, for the subshift of a primitive random substitution that is non-deterministic, there exist uncountably many frequency measures supported on this subshift \cite{gohlke-spindeler}. In contrast, the subshift of a primitive deterministic substitution has precisely one frequency measure, which is the unique ergodic measure \cite{queffelec}.

\section{Measure theoretic entropy}\label{S MT entropy}

If $T$ is an invertible measure preserving transformation of a probability space $(X, \mathcal{B}, \nu)$ and if $\xi$ is a finite measurable partition of $X$ with $\bigvee_{i \in \mathbb{Z}} T^{-i} (\xi) = \mathcal{B}$, up to null sets, then we define the \textsl{entropy} $h(T, \nu)$ of $\nu$ with respect to $T$ by
	\begin{align*}
	h(T, \nu) = \lim_{n \to \infty} \frac{1}{2n} \sum_{A \in \xi_{n}} \!-\nu(A) \log(\nu(A)),
	\end{align*}
where $\xi_{k} = \bigvee_{i = -k}^{k-1} T^{-i} (\xi)$ for $k \in \mathbb{N}$. 
In the case when $X \subseteq \mathcal{A}^{\mathbb{Z}}$ is a subshift and $\nu$ is an $S$-invariant probability measure supported on $X$, it is known that, for $m \in \mathbb{N}$,
	\begin{align*}
	h(S^{m}, \nu) = \lim_{n \to \infty} \frac{1}{n} \sum_{u \in \mathcal{L}^{mn}(X)} \hspace{-1em}-\nu([u]) \log(\nu([u])) = m h(S, \nu), 
	\end{align*}
where $\mathcal{L}^{k}(X)$ denotes the set of all words of length $k$ in the language $\mathcal{L}(X)$ of $X$, for $k$ a natural number. 
Since a primitive random substitution $\vartheta_{\mathbf{P}} = (\vartheta, \mathbf{P})$ satisfies $\mathcal{L} (X_{\vartheta}) = \mathcal{L}_{\vartheta}$, the entropy of a frequency measure $\mu_{\mathbf{P}}$ supported on $X_{\theta}$ is given by
	\begin{align*}
	h(S, \mu_{\mathbf{P}}) = \lim_{n \to \infty} \frac{1}{n} \sum_{u \in \mathcal{L}_{\vartheta}^{n}} \!- \mu_{\mathbf{P}} ([u]) \log(\mu_{\mathbf{P}} ([u])).
	\end{align*}
In what follows we will predominantly be concerned with computing the invariant $h(S, \mu_{\mathbf{p}})$ and so when it is clear from the context, we write $h(\mu_{\mathbf{p}})$ for $h(S, \mu_{\mathbf{p}})$; this will be the case in all of what follows, except in the proof of \Cref{T intrinsic ergodicity}.

Two additional concepts, which we will utilise in the proof of \Cref{T intrinsic ergodicity} is the entropy and conditional entropy of a partition.  In order to define these quantities, let $\eta$ be a second measurable partition of $X$.  The \textsl{entropy} $H_{\nu}(\eta)$ of $\eta$ with respect to $\nu$ is defined to be the quantity
	\begin{align*}
	H_{\nu}(\eta) = \sum_{A \in \eta} -\nu(A) \log(\nu(A)),
	\end{align*}
and we note that, by Fekete's Lemma,
	\begin{align}\label{eq:refence_page_20_i}
	h(T, \nu) = \inf_{n \in \mathbb{N}} \frac{1}{2n} H_{\nu}(\xi_{n}).
	\end{align}
The \textsl{entropy of $\xi$ given $\eta$} with respect to $\nu$ is defined by
	\begin{align*}
	H_{\nu}(\xi \, \vert \, \eta) 
	= - \sum_{A \in \eta} \nu(A) H_{\nu^{}_A}(\xi),
	\end{align*}
where $\nu^{}_A \colon B \mapsto \nu(A \cap B)/\nu(A)$ denotes the normalized restriction of $\nu$ to the set $A$.

We will mostly be concerned with partitions that are generated by some random map $\mc U$, that is, a measurable function on a probability space $(\Omega,\mc F,\nu)$. More precisely, if $\mc U$ has a finite image $\Ima(\mc U)$ (i.e. if it takes only finitely many values), it generates the partition
\[
\xi(\mc U) = \{ \mc U^{-1}(u) : u \in\Ima(\mc U) \}.
\]
To avoid heavy notation, we set 
\[
H_{\nu}(\mc U) : = H_{\nu}(\xi(\mc U)),
\]
in such situations. If we are dealing with two such random maps $\mc U$ and $V$, we set
\[
H_{\nu}(\mc U, \mc V) : = H_{\nu}(\xi((\mc U, \mc V)))
\]
where
\[
\xi((\mc U, \mc V)) = \xi(\mc U) \vee \xi(\mc V) := \{ A \cap B : A \in \xi(\mc U), B \in \xi(\mc V) \},
\]
is a common refinement of the partitions generated by $\mc U$ and $\mc V$. Conditional entropies are defined accordingly. Namely, $H_{\nu} (\mc U \, \vert \, \mc V) = H_{\nu} (\xi (\mc U) \, \vert \, \xi (\mc V))$, $H_{\nu} (\mc U, \mc V \, \vert \, \mc W) = H_{\nu} (\xi (\mc U) \vee \xi (\mc V) \, \vert \, \xi (\mc W))$ and $H_{\nu} (\mc U \, \vert \, \mc V, \mc W) = H_{\nu} (\xi (\mc U) \, \vert \, \xi (\mc V) \vee \xi (\mc W))$, where $\mc U, \mc V$ and $\mc W$ are random maps on $(\Omega, \mathcal{F}, \nu)$.

In the proof of our main results, we will freely use several properties of (conditional) entropy. For the reader's convenience we list the most important ones in the following; compare \cite[Ch.~4]{MR648108}.

\begin{lemma}
Let $\mc U, \mc V$ and $\mc W$ be (measurable) random maps with finite image as above. Then,
\begin{enumerate}
\item $H_{\nu}(\mc U) \leq \log(\# \Ima(\mc U))$, with equality precisely if $\nu \circ \mc U^{-1}$ is equi-distributed.
\item $H_{\nu}(\mc U) \leq H_{\nu}(\mc U, \mc V)$, with equality precisely if $\mc U$ determines $\mc V$ (up to nullsets).
\item $H_{\nu} (\mc U, \mc V ) = H_{\nu}(\mc V) + H_{\nu}(\mc U \, \vert \, \mc V)$.
\item $H_{\nu}(\mc U \, \vert \, \mc V) \leq H_{\nu}(\mc U)$, with equality if and only if $\mc U$ and $\mc V$ are independent.
\item $H_{\nu} (\mc U \, \vert \, \mc V, \mc W) \leq H_{\nu} (\mc U \, \vert \, \mc V)$.
\item $H_{\nu} (\mc U, \mc V \, \vert \, \mc W) = H_{\nu} (\mc U \, \vert \, \mc W) + H_{\nu} (\mc V \, \vert \, \mc U, \mc W)$.
\end{enumerate}
\end{lemma}
We refer the reader to \cite{MR1963683,MR648108} for further details concerning the entropy of a measure preserving transformation and that of a partition.

\subsection{Main results}\label{subsec:inflation_word_entropy}

The aim of this section is to relate the entropy of the frequency measure $\mu_{\mathbf{P}}$ to a sequence of entropy vectors which are related to inflation words $\vartheta^n(a)$ with $n \in \N$ and $a \in \mc A$. This will establish a natural analogue to the results on topological entropy presented in Section~\ref{SS top entropy}. However, we emphasise that our present setting is more general as we do not require the random substitution to be compatible. We make the standing assumption that $\vartheta_{\mathbf{P}}$ is a \emph{primitive} random substitution throughout.

\begin{definition}
For a primitive random substitution $\vartheta_{\mathbf{P}}$ on $\mathcal A$ and $m \in \N$, we let $\mathbf{H}_m = (H_{m,a})_{a \in \mathcal A}$ denote the vector with entries $H_{m,a} = H_{\mathbb{P}}(\vartheta_{\mathbf{P}}^m(a))$ for all $a \in \mathcal A$.
\end{definition}

As a further notational tool, we write $H(p)$ for the entropy of the vector $(p,1-p)$, that is,
\[
H(p) = - p \log(p) - (1-p) \log (1-p).
\]
Our most general result on the relation between the entropy of $\mu_{\mathbf{P}}$ and the sequence of entropies assigned to the Markov processes $(\vartheta_{\mathbf{P}}^n(a))_{n \in \N}$, with $a \in \mc A$, takes the following form.

\begin{theorem}
\label{THM:main}
Let $\vartheta_{\mathbf{P}}$ be a primitive random substitution with  Perron--Frobenius eigenvalue $\lambda$ and right eigenvector $\mathbf{R}$. Let $\mu_{\mathbf{P}}$ be its frequency measure on $(X_{\vartheta},S)$. Then, for all $k \in \N$,
\[
\frac{1}{\lambda^k} \mathbf{H}_k^{\top} \mathbf{R} - H(\lambda^{-k}) 
\leq h(\mu_{\mathbf{P}})
\leq \frac{1}{\lambda^k - 1} \mathbf{H}_k^{\top} \mathbf{R}.
\]
In particular,
\[
h(\mu_{\mathbf{P}}) = \lim_{k \to \infty} \frac{1}{\lambda^k} \mathbf{H}_k^{\top} \mathbf{R}.
\]
\end{theorem}

In particularly convenient situations it is possible to omit the counterterm $H(\lambda^{-k})$.
This is the case if $\vartheta_{\mathbf{P}}$ has unique realisation paths, which allows us to gain more control over the bounds for the measure theoretic entropy.
Moreover, in the case when the random substitution satisfies the disjoint set condition we obtain a closed form formula.
We also obtain a closed form formula when the random substitution satisfies the identical set condition, provided the production probabilities satisfy the following condition.

\begin{definition}
Let $\vartheta_{\mathbf{P}} = (\vartheta, \mathbf{P})$ be a random substitution satisfying the identical set condition. We say that $\vartheta_{\mathbf{P}}$ has \textsl{identical production probabilities} if for all $a \in \mathcal{A}$, $k \in \mathbb{N}$ and $v \in \vartheta^{k} (a)$, we have
	\begin{align*}
	\mathbb{P} [\vartheta_{\mathbf{P}}^{k-1} (u_1) = v] = \mathbb{P} [\vartheta_{\mathbf{P}}^{k-1} (u_2) = v]
	\end{align*}
for all $u_1$ and $u_2 \in \vartheta (a)$.
\end{definition}

\begin{theorem}
\label{THM:main-urp}
Assume that $\vartheta_{\mathbf{P}}$ is a primitive random substitution with unique realisation paths, with Perron--Frobenius eigenvalue $\lambda$ and right eigenvector $\mathbf{R}$. Then, for all $k \in \N$,
\[
\frac{1}{\lambda^k} \mathbf{H}_k^{\top} \mathbf{R} \leq h(\mu_{\mathbf{P}}) \leq \frac{1}{\lambda^k - 1} \mathbf{H}_k^{\top} \mathbf{R},
\]
where the upper bound is an equality if and only if $\vartheta_{\mathbf{P}}^k$ satisfies the disjoint set condition. The lower bound is an equality if and only if $\vartheta_{\mathbf{P}}^k$ satisfies the identical set condition with identical production probabilities. Further, the sequence of lower bounds $(\lambda^{-n} \mathbf{H}^{\top}_n \mathbf{R})_{n \in \N}$ is non-decreasing in $n$.
\end{theorem}

\begin{remark}
\label{REM:DSC-ISC-reformulation}
The conditions that allow us to obtain closed expressions for the entropy in Theorem~\ref{THM:main-urp} have been formulated in a manner that parallels our discussion of topological entropy. They can also be rephrased in probabilistic terms. More precisely, $\vartheta_{\mathbf{P}}$ satisfies the disjoint set conditions if and only if $\vartheta_{\mathbf{P}}(a)$ is determined by $\vartheta_{\mathbf{P}}^n(a)$ for all $n\in \N$ and $a \in \mc A$. The identical set condition with identical production probabilities holds for $\vartheta_{\mathbf{P}}$ if and only if the random words $\vartheta_{\mathbf{P}}(a)$ and $\vartheta_{\mathbf{P}}^n(a)$ are independent for all $n \geq 2$ and $a \in \mc A$.
\end{remark}

Comparing \Cref{THM:main} and \Cref{THM:main-urp} one of the most striking differences is that the term $H(\lambda^{-k})$ does not appear in the lower bound under the assumption of unique realisation paths. It is natural to inquire whether this term can also be dropped in the more general case of primitive random substitutions. 
That this is not the case can be seen from the following example.

\begin{example}
Let $p \in (0,1)$ and let $\vartheta_{\mathbf{P}}$ be the random substitution defined by
\[
\vartheta_{\mathbf{P}} \colon \begin{cases}
 a \mapsto & \begin{cases}
 a & \mbox{ with probability } p,
\\ aba & \mbox{ with probability } 1-p,
 \end{cases}
 \\b \mapsto & bab.
 \end{cases}
\]
This random substitution gives rise to the periodic subshift $X_{\vartheta} = \{ (ab)^{\Z}, (ba)^{\Z} \}$, which has entropy $0$. On the other hand, $M$ is primitive and $H_{\mathbb{P}}(\vartheta_{\mathbf{P}}(a)) > 0$. 
\end{example}

In general, the measure theoretic entropy $h(\mu_{\mathbf{P}})$ depends on the choice of $\mathbf{P}$. As a consequence of \Cref{THM:main} we obtain that the dependence on the probability parameters is continuous. 

In the following, we regard $\mathbf{P}$ as a vector in $\mathbb{R}^{r}$ equipped with the Euclidean topology, where $r = \sum_{i = 1}^{d} r_i = \sum_{i = 1}^{d} \# \vartheta(a_i)$ and $d$ is the cardinality of the alphabet.
We emphasize that we assume that $\mathbf{P}$ is non-degenerate in the sense that all probabilities are assumed to be strictly positive.

\begin{corollary}
Assume the setting of \Cref{THM:main}. The map $\mathbf{P} \mapsto h (\mu_{\mathbf{P}})$ is continuous.
\end{corollary}

\begin{proof}
For $0< \varepsilon < 1$ let $D_{\varepsilon}$ be the domain of those $\mathbf{P}$ such that all entries of $\mathbf{P}$ are greater than $\varepsilon$. Since we get the complete domain of $\mathbf{P}$ as a (nested) union over all $D_{\varepsilon}$, it is enough to show that the map $\mathbf{P} \mapsto h(\mu_{\mathbf{P}})$ is continuous on $D_{\varepsilon}$ for arbitrary $\varepsilon$. 
The general strategy of the proof is to represent $h(\mu_{\mathbf{P}})$ as a uniform limit of continuous functions on $D_{\varepsilon}$ via \Cref{THM:main}.

Recall that all of the data $\lambda, \mathbf{H}_m, \mathbf{R}$ depend implicitly on $\mathbf{P}$. 
By primitivity $\lambda > 1$ is a \emph{simple} eigenvalue for all $\mathbf{P}$. 
Since the substitution matrix depends analytically on the probability parameters, we can resort to fundamental facts in perturbation theory; compare for example \cite{kato}. 
In particular, $\lambda$ depends analytically on $\mathbf{P} \in D_{\varepsilon}$ and since $\lambda$ is simple, so does $\mathbf{R}$. The entries of $\mathbf{H}_m$ inherit continuity from the fact that the maps $\mathbf{P} \mapsto \mathbb{P}[\vartheta^m_{\mathbf{P}}(a) = u]$ are continuous for all $a \in \mc A$ and $u \in \mc A^+$.
Hence, the function
\[
s_m \colon \mathbf{P} \mapsto \frac{1}{\lambda^m} \mathbf{H}_m^{\top} \mathbf{R},
\]
is continuous in $\mathbf{P}$ for all $m \in \N$. With this notation, \Cref{THM:main} can be rephrased as
\begin{equation}
\label{EQ:convergence-bounds}
\frac{\lambda^m - 1}{\lambda^m} h(\mu_{\mathbf{P}}) \leq s_m(\mathbf{P}) \leq h(\mu_{\mathbf{P}}) + H(\lambda^{-m}),
\end{equation}
for all $m \in \N$. Note that $h(\mu_{\mathbf{P}})$ is uniformly bounded from above by the topological entropy of $X_{\vartheta}$ and $\lambda$ is bounded from below by its minimal value $\lambda_{\varepsilon}>1$ on the compact set $D_{\varepsilon}$. Therefore, the convergence
\[
\lim_{m \to \infty} s_m(\mathbf{P}) = h(\mu_{\mathbf{P}})
\]
is uniform on $D_{\varepsilon}$ which implies the assertion.
\end{proof}

\subsection{Renormalisation}
\label{SUBSEC:renormalisation}

Properties adhering to a (deterministic) substitution subshift can often be expressed more directly in terms of the corresponding substitution. A key observation in this regard is that a substitution subshift exhibits a self-similar structure that relates it directly to the substitution action via a renormalisation step. More precisely, every sequence in the subshift can be decomposed into inflation words of type $\vartheta(a)$, with $a \in \mc A$ such that replacing $\vartheta(a)$ by $a$ gives another sequence in the subshift. This corresponds to an (average) change of the scale by a factor $\lambda$. 
In the primitive case, keeping track of letter frequencies during this procedure provides a consistency relation that immediately shows that they must form a right Perron--Frobenius eigenvector of the substitution matrix.

A similar procedure works for word frequencies, if the substitution is replaced by an induced substitution \cite{queffelec}. This can be extended to primitive random substitutions \cite[Prop.~5.8]{gohlke-spindeler}, showing that the probability distribution $\mu^{(n)}$ on $\mathcal L_{\vartheta}^n$, given by
\[
\mu^{(n)}(w) = \mu_{\mathbf{P}}([w]),
\]
is the unique normalised Perron--Frobenius eigenvector of an appropriate induced substitution matrix $M_n$, for all $n \in \N$. 
This gives the following self-consistency relation, which was shown as the first step in the proof of \cite[Prop.~5.8]{gohlke-spindeler}.

\begin{lemma}
\label{LEM:key}
Let $\vartheta_{\mathbf{P}}$ be a primitive random substitution. Then, for all $w \in \mathcal L_{\vartheta}^n$,
\[
\mu_{\mathbf{P}}([w]) = \sum_{v \in \mathcal L_{\vartheta}^n} \mu_{\mathbf{P}}([v]) \frac{1}{\lambda} \sum_{m = 1}^{|\vartheta|} \sum_{j = 1}^{m} \mathbb{P} [\vartheta_{\mathbf{P}}(v)_{[j,j+n-1]} = w \wedge |\vartheta_{\mathbf{P}}(v_1)| = m].
\]
\end{lemma}

It will be convenient to interpret the expression appearing in \Cref{LEM:key} via the distribution of an appropriate random variable that mirrors the action of $\vartheta_{\mathbf{P}}$ on the initial distribution $\mu_{\mathbf{P}}$, together with the choice of the origin in the inflation word decomposition.

\begin{lemma}
\label{LEM:random-word}
For $n \in \N$, $\mu^{(n)}$ is the distribution of a random word $\mathcal W_n$ on a finite probability space $(\Omega_n, P_n)$, defined as follows.
The space
\[
\Omega_n = \{ (v, u_1, \cdots, u_n, j) : v \in \mathcal L_{\vartheta}^n, u_i \in \vartheta(v_i), 1 \leq j \leq |u_1| \}
\]
is equipped with the probability vector
\[
P_n \colon (v, u_1, \cdots, u_n, j) \mapsto \frac{1}{\lambda} \mu_{\mathbf{P}}([v]) \prod_{i = 1}^n \mathbb{P} [\vartheta_{\mathbf{P}}(v_i) = u_i].
\]
The random word $\mathcal W_n$ is defined via
\[
\mathcal W_n \colon (v, u_1, \cdots, u_n, j) \mapsto (u_1 \cdots u_n)_{[j,j+n-1]}.
\]
\end{lemma}

\begin{proof}
Let $w \in \mathcal L_{\vartheta}^n$. We note that $\mathcal W_n^{-1}(\{w\})$ comprises all those elements in $\Omega_n$  such that the property $(u_1 \cdots u_n)_{[j,j+n-1]} = w$ holds. That is,
\[
P_n(\mc W_n = w) = \sum_{v \in \mathcal L_{\vartheta}^n} \sum_{u_1, \cdots, u_n} \sum_{j = 1}^{|u_1|} \frac{1}{\lambda} 
\mu_{\mathbf{P}} ([v]) \prod_{i = 1}^n \mathbb{P} [\vartheta_{\mathbf{P}}(v_i) = u_i] \, \delta_{w, (u_1 \cdots u_n)_{[j,j+n-1]}}.
\]
Comparing with the expression in Lemma~\ref{LEM:key}, we further note that
\[
\mathbb{P} [\vartheta_{\mathbf{P}}(v)_{[j,j+n-1]} = w \wedge |\vartheta_{\mathbf{P}}(v_1)| = m] = \sum_{u_1, \cdots, u_n} \prod_{i = 1}^n \mathbb{P} [\vartheta_{\mathbf{P}}(v_i) = u_i] \, \delta_{m,|u_1|} \, \delta_{w, (u_1 \cdots u_n)_{[j,j+n-1]}}.
\]
From this, we obtain that $P_n(\mathcal W_n = w) = \mu_{\mathbf{P}}([w])$ and the claim follows.
\end{proof}

\begin{remark}
\label{REM:V-distribution}
We may interpret the factors occurring in the definition of $P_n$ in terms of the renormalisation step. The term $\lambda^{-1}$ corresponds to a change of scale due to the expansion of the length of words, $\mu_{\mathbf{P}}([v])$ reflects the choice of a word before the inflation step, and each of $\mathbb{P}[\vartheta_{\mathbf{P}}(v_i) = u_i]$ gives the probability of mapping $v_i$ to the particular word $u_i$ as we apply the random substitution. Marginalized to (prefixes) of $v$, the distribution induced by $P_n$ and $\mu_{\mathbf{P}}$ are closely related but different in general. To be more precise, we will be interested in the random variable
\[
\mathcal V_{[1,m]}  \colon (v ,u_1 \cdots u_n, j) \mapsto v_{[1,m]}
\]
for some $m \leqslant n$. Integrating out the dependencies on $u_2,\ldots,u_n$ and $j$ in the first step, we obtain
\[
P_n(\mathcal V_{[1,m]} = v')
= \frac{1}{\lambda} \sum_{v, v_{[1,m]} = v'} \mu_{\mathbf{P}} ([v]) \sum_{u_1}  |u_1| \mathbb{P}[\vartheta_{\mathbf{P}}(v_1) = u_1]
 = \frac{1}{\lambda} \mu_{\mathbf{P}} ( [v'] ) \mathbb{E} [|\vartheta_{\mathbf{P}}(v_1)|].
\]
The additional factor $\lambda^{-1}\mathbb{E} [|\vartheta_{\mathbf{P}}(v_1)|]$ accounts for the fact that starting the inflation word decomposition of a word within some $u_1 \in \vartheta(v_1)$ is more probable if $\mathbb{E} [|\vartheta_{\mathbf{P}}(v_1)|]$ is large.
\end{remark}

\Cref{LEM:random-word} provides us with an alternative way to calculate the measure theoretic entropy that will be instrumental for the proof of our main theorems.

\begin{lemma}
\label{LEM:entropy-by-Wn}
The measure theoretic entropy $h(\mu_{\mathbf{P}})$ of $(X_{\vartheta}, S, \mu)$ satisfies
\[
h(\mu_{\mathbf{P}}) = \lim_{n \to \infty} \frac{1}{n} H_{P_n}(\mathcal W_n).
\]
\end{lemma}

\begin{proof}
Let $I_n \colon v \mapsto v$ be the identity map on $\mc L_n$.
By the definition of measure theoretic entropy,
\[
h(\mu_{\mathbf{P}}) = \lim_{n \to \infty} \frac{1}{n} H_{\mu^{(n)}}(I_n).
\]
Since $\mu^{(n)} = P_n \circ W_n^{-1}$ by Lemma~\ref{LEM:random-word}, it follows that
$
H_{\mu^{(n)}}(I_n) = H_{P_n}(\mathcal W_n).
$
\end{proof}

\subsection{Control over large deviations}

A useful property of any primitive random substitution  $\vartheta_{\mathbf{P}}$ is that its Perron--Frobenius eigenvalue $\lambda$ can be regarded as an inflation factor. In the case that $\vartheta_{\mathbf{P}}$ is of constant length $\ell$, this interpretation is exact in the sense that $|\vartheta(v)| = \ell |v|$ for all $v \in \mc A^+$ and all realisations of $\vartheta_{\mathbf{P}}(v)$. If $\vartheta_{\mathbf{P}}$ is compatible, $|\vartheta(v)|$ is still independent of the realisation but might deviate slightly from $\lambda |v|$. However, we still obtain that $\lambda$ is arbitrarily close to the actual ratio $|\vartheta(v)|/|v|$ for large enough values of $|v|$. This is a consequence of the following result on the length of inflation words which is a mild adaptation of \cite[Proposition~5.8]{queffelec} and hence given without proof.

\begin{lemma}\label{L letter frequencies}
Let $\vartheta_{\mathbf{P}} = (\vartheta, \mathbf{P})$ be a primitive random substitution that is compatible. Then, given an $\varepsilon > 0$, there exists $n_{0} \in \mathbb{N}$ such that for all $v \in \mathcal{A}^{+}$ with $\lvert v \rvert > n_{0}$,
	\begin{align*}
	\lvert v \rvert (\lambda - \varepsilon) < \lvert \vartheta(v) \rvert < \lvert v \rvert  (\lambda + \varepsilon) \text{.}
	\end{align*}
Moreover, letting $\tau$ denote the modulus of the second largest eigenvalue of $M_{\vartheta}$, there exists a constant $D > 0$ so that, for all $i \in \{ 1, 2, \ldots, d \}$ and  $m \in \mathbb{N}$,
	\begin{align*}
	\lambda^{m} L_{i} - D \tau^{m} \leq \lvert \vartheta^{m}(a_{i}) \rvert \leq \lambda^{m} L_{i} + D \tau^{m}.
	\end{align*}
\end{lemma}

In general, such a strong statement does not hold if we drop the assumption of compatibility. However, the probability that $|\vartheta_{\mathbf{P}}(v)|$ deviates by a positive fraction from $\lambda |v|$ decays quickly with $|v|$ for typical choices of $v$. We will make this more precise in the following lemma in a form that is useful for our purposes.

\begin{lemma}
\label{LEM:A_n-convergence}
Let $\lambda_- < \lambda < \lambda_+$ and for each $n \in \N$ fix a positive number $m = m(n) < n$ such that $\lim_{n \to \infty} m(n) = \infty$. Further, let
\[
A_n = \{ (v, u_1,\ldots, u_n, j) : \lambda_- m \leq |u_2 \cdots u_m| \leq \lambda_+ m \},
\]
for all $n \in \N$.
Then, $\lim_{n \to \infty} P_n(A_n) = 1$.
\end{lemma}

\begin{proof}
Let $A_n^u := \{ (u_2,\ldots,u_m) \colon \lambda_- m \leq |u_2 \cdots u_m| \leq \lambda_+ m \}$ be the set of $(u_2,\ldots,u_m)$-tuples that extend to elements in $A_n$. By definition of $P_n$ and $A_n$,
\begin{align*}
P_n(A_n) & = \frac{1}{\lambda} \sum_{v_{[1,m]}} \mu_{\mathbf{P}} ([ v_{[1,m]} ]) \sum_{u_1} |u_1| \mathbb{P} [ \vartheta_{\mathbf{P}}(v_1) = u_1] \sum_{ (u_2,\ldots,u_m) \in A_n^u} \, \prod_{i = 1}^m \mathbb{P}[\vartheta_{\mathbf{P}}(v_i) = u_i]
\\ & = \frac{1}{\lambda} \sum_{v_{[1,m]}} \mu_{\mathbf{P}} ([v_{[1,m]}]) \mathbb{E} [|\vartheta_{\mathbf{P}}(v_1)|] 
\, \mathbb{P}[\lambda_- m \leq |\vartheta_{\mathbf{P}}(v_2 \cdots v_m)| \leq \lambda^+ m].
\end{align*}
We claim that for $\mu_{\mathbf{P}}$-almost every $v \in X_{\vartheta}$, it is
\begin{equation}
\label{EQ:as-lower-bound}
\lim_{m \to \infty} \mathbb{P}[ \lambda_- m \leq |\vartheta_{\mathbf{P}}(v_2 \cdots v_m)| \leq \lambda_+ m ] = 1.
\end{equation}
This can be seen as follows.  By ergodicity of $\mu_{\mathbf{P}}$, for $\mu_{\mathbf{P}}$-almost every $v$ and every given $\delta > 0$ it holds that 
\[
 m(R_a - \delta) \leq |v_{[2,m]}|_a \leq m(R_a + \delta),
\]
for each $a \in \mathcal A$ and large enough $m \in \N$. In this case, it follows by standard large deviation arguments (see for example \cite{denHollander}) that for all $\delta'>0$,
\begin{equation}
\label{EQ:inflation-word-length-bound}
\sum_{i,v_i = a} |\vartheta_{\mathbf{P}}(v_i)| \leq (1 + \delta') m (R_a + \delta) \mathbb{E}[|\vartheta_{\mathbf{P}}(a)|],
\end{equation}
up to a set  $E = E(m,v,\delta,\delta')$ whose probability decays exponentially with $m$. By the definition of the substitution matrix $M$, we have
\[
\mathbb{E}[|\vartheta_{\mathbf{P}}(a)|]
= \sum_{b \in \mc A} \mathbb{E}[|\vartheta_{\mathbf{P}}(a)|_b] = \sum_{b \in \mc A} M_{ba}.
\]
Summing over $a \in \mathcal A$ in \eqref{EQ:inflation-word-length-bound}, we obtain that
\[
|\vartheta_{\mathbf{P}}(v_2 \cdots v_m)| \leq m (1 + \delta') \biggl( \sum_{a,b \in \mathcal A} M_{ba} R_a + \delta |\vartheta| \biggr)
= m (1 + \delta')(\lambda + \delta |\vartheta|),
\]
up to an exponentially decaying probability.
Choosing $\delta, \delta'$ small enough, we get 
\[
|\vartheta_{\mathbf{P}}(v_2 \cdots v_m)| \leq \lambda_+ m
\]
in these cases. The estimate for the lower bound works by completely analogous arguments.
Hence, we have that there exists $c = c(v) > 0$ and $m_0 = m_0(v)$ such that
\[
\mathbb{P}[\lambda_- m \leq |\vartheta_{\mathbf{P}}(v_2 \cdots v_m)| \leq \lambda_+ m ] \geq 1 - \me^{- m c},
\]
for all $m \geq m_0$. In particular, \eqref{EQ:as-lower-bound} holds $\mu_{\mathbf{P}}$-almost surely and we get by dominated convergence,
\[
\lim_{n \to \infty} P_n(A_n) = \frac{1}{\lambda} \int_{X_{\vartheta}} \mathbb{E}[|\vartheta_{\mathbf{P}}(v_1)|] \dd \mu_{\mathbf{P}}(v) = 1.
\]
\end{proof}

\subsection{The upper bound}

As a first step towards the proof of our main theorems, we establish the sequence of upper bounds for the measure theoretic entropy that we stated in \Cref{THM:main} and \Cref{THM:main-urp}. 
For ease of notation, we let $\varphi$ denote the function 
\[
\varphi \colon x \mapsto - x \log(x),
\] 
for positive $x \in \mathbb{R}$, and set $\varphi(0) = 0$.  
To handle various terms that are of no concern for the main calculations, we also recall some standard notation on error terms. Given a positive function $f \colon \N \to \R$, we denote by $O(f)$ any function $g \colon \N \to \R$ such that $g(n)/f(n)$ is bounded in $n$. Similarly, we write $o(f)$ for a function $g \colon \N \to \R$ such that $g(n)/f(n)$ converges to $0$ as $n \to \infty$.  

\begin{prop} \label{PROP:upper-bound}
Let $\vartheta_{\mathbf{P}}$ be a primitive random substitution. Then,
\[
h(\mu_{\mathbf{P}}) \leq \frac{1}{\lambda^k-1} \mathbf{H}_k^{\top} \mathbf{R},
\]
for all $k \in \N$.
\end{prop}

\begin{proof}
It suffices to show the relation for $k = 1$, since $\mu_{\mathbf{P}}$ remains the same measure for all powers of $\vartheta_{\mathbf{P}}$. 
By Lemma~\ref{LEM:entropy-by-Wn}, it is possible to control $h(\mu_{\mathbf{P}})$ via the entropy of $\mathcal W_n$. We wish to refer to data in $\Omega_n$ via a set of appropriate random variables. To this end we introduce (or recall in the case of $\mathcal{V}_{[1,m]}$)
\begin{itemize}
\item $\mathcal V_{[1,m]}  \colon (v ,u_1 \cdots u_n, j) \mapsto v_{[1,m]}$ for all $1 \leq m \leq n$,
\item $\mathcal J \colon (v, u_1, \ldots, u_n, j) \mapsto j$,
\item $\mathcal U_k \colon (v, u_1 \cdots u_n,j) \mapsto u_k$ for all $1 \leq k \leq n$,
\item $\mathcal U_{[k,\ell]} = (\mathcal U_k, \ldots, \mathcal U_{\ell})$ for $1 \leq k \leq \ell \leq n$.
\end{itemize}
Also recall that $\mathcal W_n$ is given by $(u_1 \cdots u_n)_{[j,j+n-1]}$. 
On average, the words $u_k$ have length $\lambda$, and therefore, in typical situations, $\mathcal W_n$ in fact only depends on $u_k$ with $1 \leq k \leq m(n) $, with $m(n) \approx n / \lambda$. This motivates the following notation. Fix a small $\varepsilon > 0$ and let $\lambda_- = \lambda - \varepsilon$. Further, let $n \in \N$ and
\[
m = m_+(n) = \Bigl\lceil\frac{n}{\lambda_-}\Bigr\rceil.
\]
As a first step, we bound the entropy by
\[
H_{P_n}(\mathcal W_n) \leq H_{P_n} (\mathcal U_{[1,m]},\mathcal J) + H_{P_n} (\mathcal W_n \, | \, \mathcal U_{[1,m]}, \mathcal J).
\]
Setting 
\[
A_n = \{ (v,u_1, \ldots, u_n, j) \in \Omega_n \mid |u_2\cdots u_m| \geq n \}.
\]
we note that on $A_n$, $\mathcal W_n$ is given by $(u_1 \cdots u_m)_{[j, j+n-1]}$ and hence is completely determined by $\mathcal U_{[1,m]}$ and $\mathcal J$. On $A_n^C$, we can bound the (conditioned) entropy of $\mathcal W_n$ by
\[
\log (\# \mathcal L_{\vartheta}^n) \leq n \log(\# \mathcal A).
\]
With these two observations, we get
\[
H_{P_n} (\mathcal W_n \, | \, \mathcal U_{[1,m]}, \mathcal J ) \leq P_n(A_n^C) \, n \log(\# \mathcal A).
\]
By Lemma~\ref{LEM:A_n-convergence}, the term $P_n(A_n^C)$ converges to $0$ as $n \to \infty$ and hence
\[
H_{P_n}(\mathcal W_n) \leq H_{P_n}(\mathcal U_{[1,m]}, \mathcal J) + o(n).
\]
On the other hand, since both $\mathcal J$ and $\mathcal U_1$ have a bounded number of realisations,
\[
H_{P_n}(\mathcal U_{[1,m]}, \mathcal J) = H_{P_n}(\mathcal U_{[2,m]}) + O(1).
\]
Conditioning on $\mathcal V_{[1,m]}$, we therefore get
\begin{equation}
\label{EQ:entropy-UB-splitting}
H_{P_n}(\mathcal W_n) \leq H_{P_n}(\mathcal U_{[2,m]}) + o(n) \leq H_{P_n}(\mathcal V_{[1,m]}) + H_{P_n}(\mathcal U_{[2,m]} \, | \, \mathcal V_{[1,m]}) + o(n).
\end{equation}
For the calculation of the entropy $H_{P_n} (\mathcal V_{[1,m]})$, recall from \Cref{REM:V-distribution} that
\begin{equation}
\label{EQ:P_n-v}
P_n(\mathcal V_{[1,m]} = v_{[1,m]}) = \frac{1}{\lambda} \mu_{\mathbf{P}} ( [v_{[1,m]}] ) \mathbb{E} [|\vartheta_{\mathbf{P}}(v_1)|].
\end{equation}
 In the following, we will convince ourselves that the modification by the factor $\lambda^{-1}  \mathbb{E} [|\vartheta_{\mathbf{P}}(v_1)|]$ is inessential for our purposes. To this end, we make use of the general observation that $\varphi(pq) = p \varphi(q) + q \varphi(p)$. For an arbitrary probability vector $(p_i)_{i \in I}$ and a finite sequence of real numbers $q = (q_i)_{i \in I}$, this implies
\[
\sum_{i \in I} \varphi(p_i q_i) \leqslant \max_{i \in I} \varphi(q_i) + \sum_{i \in I} q_i \varphi(p_i).
\]
Using this for $I = \mc L_{\vartheta}$, and the probability vector with entries $\mu_{\mathbf{P}} ( [v_{[1,m]}] )$, we obtain via \eqref{EQ:P_n-v},

\[
H_{P_n}(\mathcal V_{[1,m]}) 
= \sum_{v_{[1,m]} \in \mc L_{\vartheta}^m} \varphi \biggl( \frac{1}{\lambda} \mu_{\mathbf{P}} ( [ v_{[1,m]} ] ) \mathbb{E}[|\vartheta_{\mathbf{P}}(v_1)|]  \biggr)
= O(1) + \sum_{v_{[1,m]}\in \mc L_{\vartheta}^m} \frac{\mathbb{E}[|\vartheta_{\mathbf{P}}(v_1)|]}{\lambda} \varphi(\mu_{\mathbf{P}} ([v_{[1,m]}]) ).
\]
Recall that $m = m(n)$ implicitly depends on $n$ and note that we can rewrite
\[
\frac{1}{n} \sum_{v_{[1,m]}\in \mc L_{\vartheta}^m} \frac{\mathbb{E}[|\vartheta_{\mathbf{P}}(v_1)|]}{\lambda} \varphi(\mu_{\mathbf{P}} ([v_{[1,m]}]) )
= \frac{m}{n} \int_{X_{\vartheta}} \frac{- \log(\mu_{\mathbf{P}}([v_{[1,m]}]))}{m} \, \frac{\mathbb{E}[|\vartheta_{\mathbf{P}}(v_1)|]}{\lambda}  \dd \mu_{\mathbf{P}}(v).
\]

Due to the ergodicity of $\mu_{\mathbf{P}}$ and the Shannon-MacMillan-Breiman theorem, we have that $- \log(\mu_{\mathbf{P}}([v_{[1,m]}]) )/m$ converges to $h(\mu_{\mathbf{P}})$ 
in $L^1(X_{\vartheta},\mu_{\mathbf{P}})$ and hence we also get $L^1$-convergence for the product with an arbitrary uniformly bounded function $g$ on $X_{\vartheta}$. Applying this to $g\colon v \mapsto \lambda^{-1} \mathbb{E}[|\vartheta_{\mathbf{P}}(v_1)|]$ 
yields

\begin{equation} \label{EQ:Vm-entropy}
\lim_{n \to \infty} \frac{1}{n} H_{P_n}(
\mathcal V_{[1,m]}) 
= \frac{1}{\lambda_-} h(\mu_{\mathbf{P}}) \sum_{v_1 \in \mc A} \mu_{\mathbf{P}}([v_1]) \frac{\mathbb{E} [|\vartheta_{\mathbf{P}}(v_1)|]}{\lambda} 
= \frac{1}{\lambda_-} h(\mu_{\mathbf{P}}) \frac{1}{\lambda} \sum_{a,b \in \mathcal A} M_{ba}  R_a  
= \frac{1}{\lambda_-} h(\mu_{\mathbf{P}}).
\end{equation}

We next turn to the calculation of the conditional entropy $H_{P_n}(\mathcal U_{[2,m]} \, | \, \mathcal V_{[1,m]})$.  Denoting by $P_{n,v_{[1,m]}}$ the normalized restriction of $P_n$ to $\{ \mathcal V_{[1,m]} = v_{[1,m]} \}$, we get via straightforward calculation
\[
P_{n, v_{[1,m]}} [\mathcal U_{[2,m]} = (u_2, \cdots, u_m)] = \prod_{i = 2}^{m} \mathbb{P}[\vartheta_{\mathbf{P}}(v_i) = u_i].
\]
and thereby
\[
H_{P_{n, v_{[1,m]}}} (\mathcal U_{[2,m]}) = \sum_{i = 2}^m H_{\mathbb{P}}(\vartheta_{\mathbf{P}}(v_i)) = \mathbf{H}_1^{\top} \Phi(v_{[2,m]}).
\]
Using \eqref{EQ:P_n-v}, this yields
\begin{align*}
H_{P_n}(\mathcal U_{[2,m]} \, | \, \mathcal V_{[1,m]})
= \frac{1}{\lambda} \sum_{v_{[1,m]} \in \mc L_{\vartheta}^m} 
\mu_{\mathbf{P}} ( [v_{[1,m]}] ) \mathbb{E}[|\vartheta_{\mathbf{P}}(v_1)|] \, \mathbf{H}_1^{\top}  \Phi(v_{[2,m]}).
\end{align*}
For the corresponding asymptotic behaviour we note that, again by ergodicity of $\mu_{\mathbf{P}}$, $\Phi(v_{[2,m]})/m$ converges to $\mathbf{R}$ for $\mu_{\mathbf{P}}$-almost every $v$. Thus, 
\begin{equation} \label{EQ: Um-Vm-entropy}
\lim_{n \to \infty} \frac{1}{n} H_{P_n}(\mathcal U_{[2,m]} \, | \, \mathcal V_{[1,m]})
= \frac{1}{\lambda_-} \mathbf{H}_1^{\top} \mathbf{R} \sum_{v_1 \in \mc A} \mu_{\mathbf{P}} ([v_1]) \frac{\mathbb{E}[|\vartheta_{\mathbf{P}}(v_1)|]}{\lambda} = \frac{1}{\lambda_-} \mathbf{H}_1^{\top} \mathbf{R}.
\end{equation}
Hence, combining the contributions from \eqref{EQ:Vm-entropy} and \eqref{EQ: Um-Vm-entropy}, we get by \eqref{EQ:entropy-UB-splitting},
\[
h(\mu_{\mathbf{P}}) = \lim_{n \to \infty} \frac{1}{n} H_{P_n}(\mathcal W_n) \leq \frac{1}{\lambda_-} \bigl( h(\mu_{\mathbf{P}}) + \mathbf{H}_1^{\top} \mathbf{R}). 
\]
As $\varepsilon \to 0$, we obtain $\lambda_- \to \lambda$ and hence
\[
h(\mu_{\mathbf{P}}) \leq \frac{1}{\lambda - 1} \mathbf{H}_1^{\top} \mathbf{R},
\]
completing the proof.
\end{proof}

The sequence of vectors $(\mathbf{H}_n^{\top} )_{n \in \N}$ can be bounded via a matrix-recursion that involves the substitution matrix.

\begin{prop} \label{PROP:upper-iteration}
Let $\vartheta_{\mathbf{P}}$ be a primitive random substitution. Then, for every $n,k \in \N$, we have that
\[
\mathbf{H}_{n+k}^{\top} \leq \mathbf{H}_n^{\top} M^k + \mathbf{H}_k^{\top},
\]
to be understood elementwise. In particular,
\[
\mathbf{H}_{n+k}^{\top} \mathbf{R} \leq \lambda^k \mathbf{H}_n^{\top} \mathbf{R} + \mathbf{H}_k^{\top} \mathbf{R}.
\]
If $\vartheta_{\mathbf{P}}$ has unique realisation paths, equality occurs precisely if $\vartheta_{\mathbf{P}}^{n}(a)$ is completely determined by $\vartheta_{\mathbf{P}}^{n+k}(a)$.
\end{prop}

\begin{proof}
First, let $v \in \mathcal L_{\vartheta}^m$ and note that the random variable $\vartheta_{\mathbf{P}}^n(v)$ can be written as a function of $(\vartheta_{\mathbf{P}}^n(v_1), \cdots, \vartheta_{\mathbf{P}}^n(v_m))$. Due to the independence of the random variables in the last tuple, we obtain that
\[
H_{\mathbb{P}} (\vartheta_{\mathbf{P}}^n(v)) \leq H_{\mathbb{P}} \bigl(\vartheta_{\mathbf{P}}^n(v_1), \ldots, \vartheta_{\mathbf{P}}^n(v_m) \bigr) 
= \sum_{i = 1}^m H_{\mathbb{P}}(\vartheta_{\mathbf{P}}^n(v_i)) = \mathbf{H}_n^{\top} \Phi(v).
\]
If $\vartheta_{\mathbf{P}}$ has unique realisation paths, we even obtain equality.
Using the Markov property of the substitution process in the first step, we get for every $a \in \mathcal A$,
\begin{align*}
H_{\mathbb{P}}(\vartheta_{\mathbf{P}}^{n+k}(a)| \vartheta_{\mathbf{P}}^k(a))
& = \sum_{v \in \vartheta_{\mathbf{P}}^k(a)} \mathbb{P}[\vartheta_{\mathbf{P}}^k(a) = v] H_{\mathbb{P}} (\vartheta_{\mathbf{P}}^n(v)) 
\leq \mathbf{H}_n^{\top} \sum_{v \in \vartheta_{\mathbf{P}}^k(a)} \mathbb{P}[\vartheta_{\mathbf{P}}^k(a) = v] \Phi(v)
\\&  = \mathbf{H}_n^{\top} \mathbb{E} [\Phi(\vartheta_{\mathbf{P}}^k(a))] = \mathbf{H}_n^{\top} M^k e_a,
\end{align*}
again with equality in case of unique realisation paths.
Therefore, for all $a \in \mathcal A$,
\[
H_{\mathbb{P}}(\vartheta_{\mathbf{P}}^{n+k}(a)) 
\leq H_{\mathbb{P}}(\vartheta_{\mathbf{P}}^{n+k}(a)|\vartheta_{\mathbf{P}}^k(a)) + H_{\mathbb{P}}(\vartheta_{\mathbf{P}}^k(a))
\leq \mathbf{H}_n^{\top} M^k e_a + H_{k,a}.
\]
The first inequality is an equality precisely if $\vartheta_{\mathbf{P}}^k(a)$ is completely determined by $\vartheta_{\mathbf{P}}^{n+k}(a)$ and the second inequality is an equality, provided that $\vartheta_{\mathbf{P}}$ has unique realisation paths.
\end{proof}

\begin{corollary}
\label{COR:upper-bound-recursion}
Let $\vartheta_{\mathbf{P}}$ be a primitive random substitution. Then, for all $n \in \N$,
\[
\frac{1}{\lambda^n - 1} \mathbf{H}_n^{\top} \mathbf{R} \leq \frac{1}{\lambda - 1} \mathbf{H}_1^{\top} \mathbf{R}.
\]
If $\vartheta_{\mathbf{P}}$ has unique realisation paths, we have equality for all $n \in \N$ if and only if $\vartheta_{\mathbf{P}}$ satisfies the disjoint set condition.
\end{corollary}

\begin{proof}
Given $n \geq 2$, iterating the relation $\mathbf{H}_n^{\top} \mathbf{R} \leq \lambda^{n-1} \mathbf{H}_1^{\top} \mathbf{R} + \mathbf{H}_{n-1}^{\top} \mathbf{R}$ yields
\[
\mathbf{H}_n^{\top} \mathbf{R} \leq \mathbf{H}_1^{\top} \mathbf{R} \sum_{k = 0}^{n-1} \lambda^k = 
\frac{\lambda^n -1}{\lambda - 1} \mathbf{H}_1^{\top} \mathbf{R},
\]
immediately giving the required inequality. Given the property of unique realisation paths, equality holds if and only if $\vartheta_{\mathbf{P}}^n(a)$ completely determines $\vartheta_{\mathbf{P}}(a)$ for all $a \in \mathcal A$ and $n \in \N$. This is just a reformulation of the disjoint set condition; compare \Cref{REM:DSC-ISC-reformulation}.
\end{proof}

\subsection{The lower bound}

In this section, we will establish the lower bounds for the measure theoretic entropy in \Cref{THM:main} and \Cref{THM:main-urp}. Again, our proof relies heavily on the self-consistency relation for $\mu_{\mathbf{P}}$ presented in \Cref{SUBSEC:renormalisation}.

\begin{prop}
\label{PROP:lower-bound}
Let $\vartheta_{\mathbf{P}}$ be a primitive random substitution with associated measure $\mu_{\mathbf{P}}$. Then,
\[
h(\mu_{\mathbf{P}}) \geq \frac{1}{\lambda^k} \mathbf{H}_k^{\top} \mathbf{R} - H(\lambda^{-k}),
\]
for all $k \in \N$. If $\vartheta_{\mathbf{P}}$ has unique realisation paths, it is
\[
h(\mu_{\mathbf{P}}) \geq \frac{1}{\lambda^k} \mathbf{H}_k^{\top} \mathbf{R},
\]
for all $k \in \N$.
\end{prop}

\begin{proof}
Again, it suffices to consider the case $k=1$.
We take over the notation from the proof of Proposition~\ref{PROP:upper-bound} with one modification. For $\varepsilon > 0$, we now consider $\lambda_+ = \lambda + \varepsilon$ and set
\[
m = m_-(n) = \Bigl\lceil\frac{n}{\lambda_+}\Bigr\rceil.
\]
This is to ensure that $\mathcal W_n$ and $\mathcal J$ determine $\mathcal U_2 \cdots \mathcal U_m$ on a set of large probability, given by
\[
B_n = \{ (v, u_1, \ldots, u_n,j) : |u_2 \cdots u_m| \leq n - |\vartheta|  \}.
\] 
Using standard properties of conditional entropy, we get
\begin{equation}
\label{EQ:lower-with-counter}
H_{P_n}(\mathcal W_n) \geq H_{P_n} (\mathcal W_n \, | \, \mathcal V_{[1,m]}) 
\geq H_{P_n} ( \mathcal U_{[2,m]} \, | \, \mathcal V_{[1,m]} ) - H_{P_n} (\mathcal U_{[2,m]} \, | \, \mathcal W_n).
\end{equation}
Just like in the proof of Proposition~\ref{PROP:upper-bound} it follows that
\[
\lim_{n \to \infty} \frac{1}{n} H_{P_n} ( \mathcal U_{[2,m]} \, | \, \mathcal V_{[1,m]} ) = \frac{1}{\lambda_+} \mathbf{H}_1^{\top} \mathbf{R}.
\]
It remains to find an adequate upper bound for $H_{P_n} (\mathcal U_{[2,m]} \, | \, \mathcal W_n)$. To that end, we introduce an additional random variable on $\Omega_n$ via
\[
\ell_m \colon (v, u_1, \ldots, u_n, j) \mapsto |u_2 \cdots u_m|.
\]
Next, we obtain
\begin{equation}
\label{EQ:H-relation-I}
\begin{split}
H_{P_n} (\mathcal U_{[2,m]} \, | \, \mathcal W_n)
& \leq H_{P_n} (\mathcal U_{[2,m]} \, | \, \mathcal W_n, \mathcal J, \ell_m) + H_{P_n} (\mathcal J, \ell_m \, | \mathcal W_n)
\\ &= H_{P_n} (\mathcal U_{[2,m]} \, | \, \mathcal W_n, \mathcal J, \ell_m) + O(\log(m)).
\end{split}
\end{equation}
The last step follows because the number of distinct realisations of $(\mathcal J, \ell_m)$ can be bounded from above by $|\vartheta|^2 m$. Conditioned on $\mathcal W_n, \mathcal J, \ell_m$, and provided $\ell_m \leq n- |\vartheta|$, knowledge of $\mathcal U_{[2,m]}$ is equivalent to knowledge of 
\[
|\mathcal U|_{[2,m]} \colon (v, u_1, \ldots, u_n,j) \mapsto (|u_2|, \ldots, |u_m|).
\]
Indeed, on the set $B_n$ (that is, if $\ell_m \leq n -|\vartheta|$) we observe that $\mathcal W_n, \mathcal J$ and $\ell_m$ determine the word $u_2 \cdots u_m$, such that knowing the lengths of the individual words allows us to infer $(u_2, \ldots, u_m)$. By conditioning,
\[
H_{P_n} (\mathcal U_{[2,m]} \, | \, \mathcal W_n, \mathcal J, \ell_m)  \leq H_{P_n} (|\mathcal U|_{[2,m]} \, | \, \mathcal W_n, \mathcal J, \ell_m) 
+ H_{P_n} (\mathcal U_{[2,m]} \, | \, |\mathcal U|_{[2,m]}, \mathcal W_n, \mathcal J, \ell_m).
\]
Let $M = \max_{a \in \mathcal A} \# \vartheta(a)$, implying $\# \sigma(\mathcal U_{[2,m]}) \leq M^m$. By the observations above, we can bound 
\[
H_{P_n} (\mathcal U_{[2,m]} \, | \, |\mathcal U|_{[2,m]}, \mathcal W_n, \mathcal J, \ell_m) \leq P_n(B_n^C) \, m \log(M).
\]
Since $P(B_n^C) \to 0$ as $n \to \infty$ by Lemma~\ref{LEM:A_n-convergence}, it follows that
\begin{equation}
\label{EQ:H-relation-II}
H_{P_n} (\mathcal U_{[2,m]} \, | \, \mathcal W_n, \mathcal J, \ell_m) \leq H_{P_n} (|\mathcal U|_{[2,m]} \, | \, \ell_m) 
+ o(n).
\end{equation}
If $\vartheta_{\mathbf{P}}$ has unique realisation paths, we even get that $\mathcal W_n, \mathcal J, \ell_m$ determine $\mathcal U_{[2,m]}$ completely on $B_n$, yielding
\[
H_{P_n} (\mathcal U_{[2,m]} \, | \, \mathcal W_n, \mathcal J, \ell_m) = o(n),
\]
by an analogous argument.
Given $\ell_m = \ell$, the number of possible values of $|U|_{[2,m]}$ is bounded above by the number of choices to decompose a block of length $\ell$ into $m-1$ smaller blocks, that is, by the binomial coefficient $\binom{\ell -1}{m -2}$. Using this bound on $B_n$ and the fixed bound $M^m$ on $B_n^C$, we obtain
\begin{align*}
H_{P_n} (|\mathcal U|_{[2,m]} \, | \, \ell_m) 
& \leq \sum_{\ell = m-1}^{n -|\vartheta|} P_n[\ell_m = \ell] \log {\binom{\ell -1}{m -2}} + P_n(B_n^C) \, m \log(M)
\\ & \leq \log { \binom{n}{m-2}} + o(n)
\leq n\, H ((m-2)/n) + o(n).
\end{align*}
Since we have seen in \eqref{EQ:H-relation-I} and \eqref{EQ:H-relation-II} that $H_{P_n} (|\mathcal U|_{[2,m]} \, | \, \ell_m)$ bounds $H_{P_n}(\mathcal U_{[2,m]} \, | \, \mathcal W_n)$ up to a term of order $o(n)$, we get from  \eqref{EQ:lower-with-counter} that
\begin{align*}
h(\mu_{\mathbf{P}}) 
& = \lim_{n \to \infty} \frac{1}{n} H_{P_n}(\mathcal W_n) \geq 
\lim_{n \to \infty} \frac{1}{n} H_{P_n}(\mathcal U_{[2,m]} \, | \, \mathcal V_{[1,m]}) - \limsup_{n \to \infty} H_{P_n}(\mathcal U_{[2,m]} \, | \, \mathcal W_n) 
\\& \geq \frac{1}{\lambda_+} \mathbf{H}_1^{\top} \mathbf{R} - H(\lambda_+^{-1}) 
\xrightarrow{ \varepsilon \to 0} \frac{1}{\lambda} \mathbf{H}_1^{\top} \mathbf{R} - H(\lambda^{-1}).
\end{align*}
If $\vartheta_{\mathbf{P}}$ has unique realisation paths, we have $H_{P_n}(\mathcal U_{[2,m]} \, | \, \mathcal W_n) = o(n)$, which gives the stronger bound
\[
h(\mu_{\mathbf{P}}) \geq \frac{1}{\lambda} H^1 R,
\]
in this case.
\end{proof}

For the remainder of this section, we restrict to the case of unique realisation paths.

\begin{prop} \label{PROP:lower-iteration}
Let $\vartheta_{\mathbf{P}}$ be a primitive random substitution with unique realisation paths. Then, 
\[
\mathbf{H}_{n+k}^{\top} \geq \mathbf{H}_n^{\top} M^k
\]
for all $n,k \in \N$. Equality holds if and only if $\vartheta_{\mathbf{P}}^{n+k}(a)$ is independent of $\vartheta_{\mathbf{P}}^n(a)$ for all $a \in \mathcal A$.
\end{prop}

\begin{proof}
As in the proof of Proposition~\ref{PROP:upper-iteration}, we obtain
\[
H(\vartheta_{\mathbf{P}}^{n+k}(a)) \geq H(\vartheta_{\mathbf{P}}^{n+k} | \vartheta_{\mathbf{P}}^n(a)) = \mathbf{H}_n^{\top} M^k e_a,
\]
for all $a \in \mathcal A$. Equality holds if and only if $\vartheta_{\mathbf{P}}^{n+k}(a)$ and $\vartheta_{\mathbf{P}}^n(a)$ are independent random variables.
\end{proof}

\begin{corollary}
\label{COR:lower-bound-recursion}
Let $\vartheta_{\mathbf{P}}$ be primitive with unique realisation paths. Then, for all $m \leq n$,
\[
\frac{1}{\lambda^m} \mathbf{H}_m^{\top} \mathbf{R} \leq \frac{1}{\lambda^n} \mathbf{H}_n^{\top} \mathbf{R}.
\]
Equality holds for all $m \leq n$ if and only if $\vartheta_{\mathbf{P}}$ satisfies the identical set condition with identical production probabilities. 
\end{corollary}

\begin{proof}
By Proposition~\ref{PROP:lower-iteration}, we get
\[
\frac{1}{\lambda^n} \mathbf{H}_n^{\top} \mathbf{R} \geq \frac{1}{\lambda^n} \mathbf{H}_m^{\top} M^{n-m} \mathbf{R} = \frac{1}{\lambda^m} \mathbf{H}_m^{\top} \mathbf{R}.
\]
Equality for all $m \leq n$ holds precisely if 
\[\
\frac{1}{\lambda^n} \mathbf{H}_n^{\top} \mathbf{R} = \frac{1}{\lambda} \mathbf{H}_1^{\top} \mathbf{R},
\]
for all $n \in \N$. This is the case if and only if for all $a \in \mathcal A$, $\vartheta_{\mathbf{P}}(a)$ is independent from $\vartheta_{\mathbf{P}}^n(a)$ for all $n \in \N$, which means that $\vartheta_{\mathbf{P}}^{n-1}(v)$ has the same distribtution for all possible realisations $v$ of $\vartheta_{\mathbf{P}}(a)$. This is precisely the identical set condition with identical production probabilities. 
\end{proof}

With the results established thus far, our main results follow in a straightforward manner.

\begin{proof}[Proof of \Cref{THM:main}]
The fact that $\lambda^{-k} \mathbf{H}_k^{\top} \mathbf{R} - H(\lambda^{-k}) \leq h(\mu_{\mathbf{P}}) \leq (\lambda^k - 1)^{-1} \mathbf{H}_k^{\top} \mathbf{R}$ for all $k \in \N$ follows directly by combining \Cref{PROP:upper-bound} and \Cref{PROP:lower-bound}. The convergence 
of $\lambda^{-k} \mathbf{H}_k^{\top} \mathbf{R}$ as $k \to \infty$ can be seen from the reformulation of this relation in \eqref{EQ:convergence-bounds}.
\end{proof}

\begin{proof}[Proof of \Cref{THM:main-urp}]
The upper and lower bounds for $h(\mu_{\mathbf{P}})$ were established in \Cref{PROP:upper-bound} and \Cref{PROP:lower-bound}. The statements on the equivalent conditions for equality with the lower or upper bound are given in \Cref{COR:upper-bound-recursion} and \Cref{COR:lower-bound-recursion}. The fact that the sequence of lower bounds is non-decreasing is also contained in \Cref{COR:lower-bound-recursion}.
\end{proof}

\section{Measures of maximal entropy}\label{S MME}

\subsection{Existence of frequency measures of maximal entropy}\label{SS MME existence results}

By comparing the results for measure theoretic entropy established in Section \ref{S MT entropy} with the results on topological entropy obtained in \cite{gohlke}, we ascertain that for random substitution subshifts there often exists a frequency measure of maximal entropy. In particular, as a consequence of \Cref{C IS DS topological entropy} and \Cref{THM:main-urp}, we obtain that every subshift of a primitive and compatible random substitution satisfying the identical set condition or disjoint set condition has a frequency measure of maximal entropy. This measure of maximal entropy is the frequency measure corresponding to uniform probabilities.

\begin{theorem}\label{T IS DS MME}
Let $\vartheta_{\mathbf{P}} = (\vartheta, \mathbf{P})$ be a primitive and compatible random substitution satisfying either the disjoint set condition or the identical set condition, with corresponding frequency measure $\mu_{\mathbf{P}}$.  If $\mathbb{P} [\vartheta_{\mathbf{P}} (a) = s] = 1/(\# \vartheta (a))$ for all $a \in \mathcal{A}$ and $s \in \vartheta (a)$, then $\mu_{\mathbf{P}}$ is a measure of maximal entropy for the system $(X_{\vartheta}, S)$.
\end{theorem}

\begin{proof}
For $a \in \mathcal{A}$ and $s \in \vartheta (a)$, we have that $\mathbb{P} [\vartheta_{\mathbf{P}} (a) = s] = 1/(\#\vartheta (a))$; hence,
	\begin{align*}
	\mathbf{H}_1^{\top} \mathbf{R} = \sum_{a \in \mathcal{A}} R_{a} \log(\# \vartheta (a)).
	\end{align*}
If $\vartheta_{\mathbf{P}}$ satisfies the disjoint set condition, then by \Cref{THM:main-urp}, we have
	\begin{align*}
	h (\mu_{\mathbf{P}}) = \frac{1}{\lambda-1} \sum_{a \in \mathcal{A}} R_{a} \log(\# \vartheta (a)).
	\end{align*}
Thus, it follows by \Cref{C IS DS topological entropy} that $h (\mu_{\mathbf{P}}) = h_{\text{top}} (X_{\vartheta})$, and so $\mu_{\mathbf{P}}$ is a measure of maximal entropy.

Assume that $\vartheta_{\mathbf{P}}$ satisfies the identical set condition. Before we can apply \Cref{THM:main-urp}, we first need to verify that $\vartheta_{\mathbf{P}}$ has identical production probabilities.  To this end, let $a \in \mathcal{A}$, and $u$ and $v \in \vartheta (a)$.   Since $\vartheta_{\mathbf{P}}$ is compatible, $\lvert u \rvert_{b} = \lvert v \rvert_{b}$ for all $b \in \mathcal{A}$.  Hence, if $t \in \vartheta^{2} (a)$, it follows that
	\begin{align*}
	\mathbb{P} [\vartheta_{\mathbf{P}} (u) = t] = \prod_{b \in \mathcal{A}} (\# \vartheta (b))^{-\lvert u \rvert_{b}} = \prod_{b \in \mathcal{A}} (\# \vartheta (b))^{-\lvert v \rvert_{b}} = \mathbb{P} [\vartheta_{\mathbf{P}} (v) = t].
	\end{align*}
By way of induction, let $n \in \mathbb{N}$ and assume that $\mathbb{P} [\vartheta_{\mathbf{P}}^{n-1} (u) = w] = [\vartheta_{\mathbf{P}}^{n-1} (v) = w]$ for all $w \in \vartheta^{n} (a)$. Since $\vartheta_{\mathbf{P}}$ satisfies the identical set condition, for all $t \in \vartheta^{n+1} (a)$ we have $t \in \vartheta^{n} (u) \cap \vartheta^{n} (v)$, so
	\begin{align*}
	\mathbb{P} [\vartheta_{\mathbf{P}}^{n} (u) = t]
	= \hspace{-1em} \sum_{w \in \vartheta^{n-1} (u)} \hspace{-0.75em} \mathbb{P} [\vartheta_{\mathbf{P}}^{n-1} (u) = w] \, \mathbb{P} [\vartheta_{\mathbf{P}} (w) = t]
	= \hspace{-1em} \sum_{w \in \vartheta^{n-1} (v)} \hspace{-0.75em} \mathbb{P} [\vartheta_{\mathbf{P}}^{n-1} (v) = w] \, \mathbb{P} [\vartheta_{\mathbf{P}} (w) = t]
	= \mathbb{P} [\vartheta_{\mathbf{P}}^{n} (v) = t].
	\end{align*}
Therefore, by induction, $\vartheta_{\mathbf{P}}$ has identical production probabilities, and thus, by \Cref{THM:main-urp}, we have
	\begin{align*}
	h(\mu_{\mathbf{P}}) = \frac{1}{\lambda} \sum_{a \in \mathcal{A}} R_{a} \log(\# \vartheta (a)).
	\end{align*}
This with \Cref{C IS DS topological entropy} yields that $h (\mu_{\mathbf{P}}) = h_{\text{top}} (X_{\vartheta})$. Namely, $\mu_{\mathbf{P}}$ is a measure of maximal entropy.
\end{proof}

In general, a primitive and compatible random substitution with uniform probabilities need not give rise to a frequency measure of maximal entropy:\ see, for instance, \Cref{EX: R Fib}. However, for any subshift of a primitive and compatible random substitution, a measure of maximal entropy can be realised as a weak limit of frequency measures.

\begin{theorem}\label{T weak limit MME}
Let $X$ be the subshift of a primitive and compatible random substitution. There exists a sequence of frequency measures $(\mu_{n})_{n}$ such that $\mu_{n}$ converges weakly to a measure of maximal entropy $\mu$ for the system $(X, S)$.
\end{theorem}

\begin{proof}
Let $\vartheta_{\mathbf{P}} = (\vartheta, \mathbf{P})$ be a primitive and compatible random substitution that gives rise to the subshift $X_{\vartheta}$, and let $\lambda$ denote the Perron--Frobenius eigenvalue of the substitution matrix $M_{\vartheta}$.  Then, for all $n \in \mathbb{N}$, the substitution $\vartheta^{n}$ gives rise to the same subshift as $\vartheta$, namely $X_{\vartheta}$.  Let $\mathbf{P}_{n}$ denote the family of probability vectors corresponding to uniform probabilities on $\vartheta^{n}$, and let $\mu_{n}$ denote the frequency measure corresponding to the random substitution $(\vartheta^{n},\mathbf{P}_{n})$. Since the space of probability measures supported on a compact set and endowed with the weak topology is compact, there exists a probability measure $\mu$ and a subsequence $(n_k)_k$ of the natural numbers such that $(\mu_{n_k})_{k \in \mathbb{N}}$ converges weakly to $\mu$. By \Cref{THM:main-urp}, we have
	\begin{align*}
	h (\mu_{n_k}) \geq \frac{1}{\lambda^{n_k}} \sum_{a \in \mathcal{A}} R_{a} \log(\# \vartheta^{n_k} (a))
	\end{align*}
for all $k \in \mathbb{N}$. By \Cref{T topological inflation entropy}, the right hand term converges to the topological entropy of the system $(X, S)$ as $k$ tends to infinity.
Hence,
	\begin{align*}
	\limsup_{k \rightarrow \infty} h (\mu_{n_k}) \geq h_{\text{top}} (X_{\vartheta}),
	\end{align*}
and so it follows, by the upper semi-continuity of measure theoretic entropy, that $h (\mu) = h_{\text{top}} (X_{\vartheta})$.
\end{proof}

\subsection{Intrinsic ergodicity}\label{SS intrinsic ergodicity}

For a class of primitive random substitutions satisfying the disjoint set condition, the frequency measure of maximal entropy given by \Cref{T IS DS MME} is the unique measure of maximal entropy among all shift-invariant Borel probability measures. This is the content of the main result of this section (\Cref{T intrinsic ergodicity}). The random substitutions considered here are all constant length and \textsl{recognisable}, the definition of which is given below. Recognisablity also appears in the work of Miro et al \cite{miro-et-al} on topological mixing of random substitutions and in Rust's paper on periodic points \cite{rust-periodic-points}. 

\begin{definition}
Let $\vartheta_{\mathbf{P}} = (\vartheta, \mathbf{P})$ denote a random substitution over a finite alphabet $\mathcal{A}$, and suppose that $\lvert  \vartheta(a) \rvert$ is well-defined for all $a \in \mathcal{A}$.  We call $\vartheta_{\mathbf{P}}$ \textsl{recognisable} if for all $x \in X_{\vartheta}$ there exist a unique $y = \cdots y_{-1}y_{0}y_{1} \cdots \in X_{\vartheta}$ and a unique integer $k \in \{ 0, \ldots, | \vartheta(y_{0}) | - 1 \}$ with $S^{-k}(x) \in \vartheta(y)$.
\end{definition}

Observe that if $\vartheta_{\mathbf{P}}$ is recognisable, then so is $\vartheta^{m}_{\mathbf{P}}$ for all $m \in \mathbb{N}$, and if $\vartheta_{\mathbf{P}}$ is of constant length $\ell$, then recognisability implies that every $x \in X_{\vartheta}$ is contained in precisely one of the sets $S^{k}(\vartheta(X_{\vartheta}))$ for $k \in \{1, \ldots, \ell \}$. Further, we have the following \textsl{local} version of recognisability. This is similar to the case of deterministic substitutions where an equivalence between global and local recognisability holds. Intuitively, local recognisability means that applying a \textsl{finite window} to a sequence is enough to determine the position and the type of the inflation word in the middle of that window.

\begin{lemma}\label{LEM:local-recog}
Let $\vartheta_{\mathbf{P}} = (\vartheta, \mathbf{P})$ denote a primitive random substitution over an alphabet $\mathcal{A}$, and suppose that $\lvert \vartheta(a) \rvert$ is well-defined for all $a \in \mathcal{A}$. If $\vartheta_{\mathbf{P}}$ is recognisable, then there exists a smallest natural number $\kappa(\vartheta)$, called the \textsl{recognisability radius of $\vartheta$}, with the following property. If $x \in \vartheta([a])$ for some $a \in \mathcal{A}$ and $x_{[-\kappa(\vartheta),\kappa(\vartheta)]} = y_{[-\kappa(\vartheta),\kappa(\vartheta)]}$ for some $y \in X_{\vartheta}$, then $y \in \vartheta([a])$.
\end{lemma}

\begin{proof}
By way of contradiction, suppose there is no radius of recognisability. In which case, there exists a sequence of tuples $((x^{(k)},y^{(k)}))_{k \in \mathbb{N}}$ with $(x^{(k)},y^{(k)}) \in \vartheta([a]) \times \vartheta( [a])^{C}$ and $x^{(k)}_{[-k,k]} = y^{(k)}_{[-k,k]}$ for all $k \in \mathbb{N}$. Let $(x,y) \in X_{\vartheta} \times X_{\vartheta}$ be an accumulation point of this sequence. By recognisability,
	\begin{align*}
	X_{\vartheta} = \bigsqcup_{b \in \mathcal{A}} \bigsqcup_{k = 0}^{\lvert \vartheta(b) \rvert - 1} S^{k} (\vartheta([b])),
	\end{align*}
and by construction, $x = y$.  Due to \Cref{L compact to compact}, and since $S$ is continuous, we have that $S^{k}(\vartheta([b]))$ is compact for all $b \in \mathcal{A}$ and $k \in \mathbb Z$. Hence, both $\vartheta([a])$ and $\vartheta([a])^{C}$ are compact. It therefore follows that $x \in \vartheta([a])$ and $x= y \in \vartheta([a])^{C}$, leading to a contradiction.
\end{proof}

\begin{lemma}\label{LEM:recog-DSC}
Assume the setting of \Cref{LEM:local-recog}. If the random substitution $\vartheta_{\mathbf{P}}$ is recognisable, then it satisfies the disjoint set condition.
\end{lemma}

\begin{proof}
By way of contradiction, suppose that $\vartheta_{\mathbf{P}}$ does not satisfy the disjoint set condition. In which case, there exist $a \in \mathcal{A}$, and $s$ and $t \in \vartheta(a)$ with $s \neq t$ and $\vartheta(s) \cap \vartheta(t) \neq \varnothing$. For $ x \in [a]$, observe that there exist $y$ and $z \in \vartheta(x)$ such that $y_{[0, \lvert \vartheta(a) \rvert -1]} = s$, $z_{[0, \lvert \vartheta(a) \rvert - 1]} = t$, and $y$ coincides with $z$ at all other positions. Hence, there exists a $w \in \vartheta(y) \cap \vartheta(z)$ that can be constructed by mapping $s$ and $t$ to the same word $v \in \vartheta(s) \cap \vartheta(t)$. This is a contradiction to recognisability.
\end{proof}

The converse of this statement does \textsl{not} hold: a counterexample is given by the random period doubling substitution. When establishing intrinsic ergodicity for certain random substitutions, we will be concerned with recognisability for some power of those random substitutions. It follows from a simple recursive argument that the recognisability radius of $\vartheta^{m}_{\mathbf{P}}$ grows (asymptotically) at most with the inflation factor as $m$ increases. For constant length substitutions the precise result reads as follows.

\begin{lemma}\label{LEM:recognisability-recursion}
Let $\vartheta_{\mathbf{P}} = (\vartheta, \mathbf{P})$ be a primitive random substitution of constant length $\ell$. If $\vartheta_{\mathbf{P}}$ is recognisable, then for all $m \in \mathbb{N}$, we have that
	\begin{align*} 
	\kappa(\vartheta^{m}) \leq \frac{\ell^{m} -1}{\ell -1} \kappa(\vartheta).
	\end{align*} 
\end{lemma} 

\begin{proof}
We proceed by induction.  The result is immediate for $m =1$. Assume it holds for some $m \in \mathbb{N}$, and note, by primitivity, that $X_{\vartheta} = X_{\vartheta^{m}}$. Let $a \in \mathcal{A}$, $x \in \vartheta^{m+1}([a])$ and $y \in X_{\vartheta}$ with $x_{[-k,k]} = y_{[-k,k]}$ for $k = \ell \kappa(\vartheta^{m}) + \kappa(\vartheta)$; in particular, $y \in \vartheta(X_{\vartheta})$. Let $v \in \vartheta^{m}([a])$ be such $x \in \vartheta(v)$, and let $w \in X_{\vartheta}$ such that $y \in \vartheta(w)$. Applying local recognisability to the pair $(S^{j\ell}x,S^{j\ell}y)$ for each $j \in \{-\kappa(\vartheta^{m}), \ldots, \kappa(\vartheta^{m}) \}$, in combination with \Cref{LEM:recog-DSC}, we obtain that $v_{[-\kappa(\vartheta^{m}), \kappa(\vartheta^{m})]} = w_{[-\kappa(\vartheta^{m}), \kappa(\vartheta^{m})]}$.  By the definition of $\kappa(\vartheta^{m})$, this implies $w \in \vartheta^{m}([a])$ and so $y \in \vartheta(w) \subseteq \vartheta^{m+1}([a])$, yielding
	\begin{align*}
	\kappa(\vartheta^{m+1}) \leq \ell \kappa(\vartheta^{m}) + \kappa(\vartheta) = \kappa(\vartheta) \sum_{j = 0}^{m} \ell^{j} = \frac{\ell^{m} - 1}{\ell - 1} \kappa(\vartheta),
	\end{align*} 
where the second to last equality follows from the inductive hypothesis.
\end{proof}

Since every primitive recognisable random substitution $\vartheta_{\mathbf{P}}$ satisfies the disjoint set condition, if $\vartheta_{\mathbf{P}}$ is compatible, then \Cref{T IS DS MME} gives that the frequency measure corresponding to uniform probabilities is a measure of maximal entropy. Without compatibility, we may not utilise Theorem \ref{T topological inflation entropy} to obtain a formula for the topological entropy of the corresponding subshift. However, we can compute directly the topological entropy for a class of random substitution subshifts that includes all the non-compatible random substitution subshifts for which we prove intrinsic ergodicity in \Cref{T intrinsic ergodicity}. This is the content of \Cref{LEM:noncompatible top entropy}. Combining this with Theorem \ref{THM:main-urp} gives that the frequency measure corresponding to uniform probabilities is a measure of maximal entropy.

\begin{lemma}\label{LEM:noncompatible top entropy}
Let $\vartheta_{\mathbf{P}}$ be a primitive recognisable random substitution of constant length $\ell$. If there exists an $N \in \mathbb{N}$ such that $\# \vartheta (a) = N$ for all $a \in \mathcal{A}$, then 
	\begin{align}\label{eq:top_entropy_constant_length}
	h_{\text{top}} (X_{\vartheta}) = \frac{1}{\ell -1} \log (N).
	\end{align}
In particular, the frequency measure $\mu$ corresponding to uniform probabilities is a measure of maximal entropy for the subshift $X_{\vartheta}$.
\end{lemma}

\begin{proof}
For $m \in \mathbb{N}$, we have
	\begin{align*}
	\mathcal{L}_{\vartheta}^{m\ell}
	= \bigcup_{v \in \mathcal{L}_{\vartheta}^{m+1}} \bigcup_{u \in \vartheta (v)} \bigcup_{j=1}^{\ell} \left\{ u_{[j,j+\ell m-1]} \right\}.
	\end{align*}
Since by our hypothesis and \Cref{LEM:recog-DSC} we have $\# \vartheta(v) = N^{\lvert v \rvert}$ for all $v \in \mathcal{L}_{\vartheta}^{m+1}$, it follows that $\# \mathcal{L}_{\vartheta}^{m\ell} \leq \ell N^{m+1} \# \mathcal{L}_{\vartheta}^{m+1}$, and so
	\begin{equation}\label{E top entropy UB}
	h_{\text{top}} (X_{\vartheta})
	= \lim_{m \rightarrow \infty} \frac{\log (\# \mathcal L_{\vartheta}^{m\ell})}{m\ell} \leq \frac{1}{\ell} \log(N) + \frac{1}{\ell} h_{\text{top}} (X_\vartheta).
	\end{equation}
On the other hand, 
	\begin{align*}
	\mathcal{L}_{\vartheta}^{m\ell} 
	\supseteq \vartheta (\mathcal{L}_{\vartheta}^{m}) 
	= \bigcup_{v \in \mathcal{L}_{\vartheta}^{m}} \vartheta (v),
	\end{align*}
By recognisability, there is a number $r \leq \kappa(\vartheta)$ such that for $u,v \in \mathcal L_{\vartheta}^m$ with $v_{[r+1, m -r]} \neq u_{[r+1, m -r]}$ we have $\vartheta(u) \cap \vartheta(v) = \varnothing$.
Hence, $\# \mathcal L_{\vartheta}^{m \ell} \geq N^m \# \mathcal L_{\vartheta}^{m - 2r}$, so
	\begin{align}\label{E top entropy LB}
	h_{\text{top}} (X_{\vartheta})
	 = \lim_{m \to \infty} \frac{1}{m \ell} \log(\# \mathcal L_\vartheta^{m\ell})
	 \geq \frac{1}{\ell} \log(N) + \frac{1}{\ell} h_{\text{top}} (X_{\vartheta}).
	\end{align}
Combining \eqref{E top entropy UB} and \eqref{E top entropy LB} and rearranging yields \eqref{eq:top_entropy_constant_length}. To see that $\mu$ is a measure of maximal entropy for the subshift $X_{\vartheta}$, observe that by \Cref{THM:main-urp} and \Cref{LEM:recog-DSC}, we have
\begin{equation*}
    h(\mu) = \frac{1}{\ell - 1} \sum_{a \in \mathcal{A}} R_{a} \sum_{s \in \vartheta (a)} \frac{1}{N} \log (N) = \frac{1}{\ell-1} \log (N) \text{,}
\end{equation*}
since $\# \vartheta (a) = N$ for all $a \in \mathcal{A}$ and $\sum_{a \in \mathcal{A}} R_{a} = 1$. Hence $h (\mu) = h_{\text{top}} (X_{\vartheta})$.
\end{proof}

\begin{remark}
In contrast to the compatible case, it is not true in general that for a primitive and constant length random substitution the measure corresponding to uniform probabilities is a measure of maximal entropy. We present an example of such a random substitution in \Cref{EX: R dependent on probs}.
\end{remark}

We now give the statement of the main result of this section, \Cref{T intrinsic ergodicity}.

\begin{theorem}\label{T intrinsic ergodicity}
Let $\vartheta_{\mathbf{P}} = (\vartheta, \mathbf{P})$ be a primitive recognisable random substitution of constant length $\ell$ and assume that at least one of the following holds:
	\begin{itemize}
	\item[(i)] $\vartheta(a)$ has the same cardinality for all $a \in \mathcal{A}$;
	\item[(ii)] $\vartheta_{\mathbf{P}}$ is compatible and $\ell$ is the only non-zero eigenvalue of the substitution matrix.
	\end{itemize}
Under these hypotheses, the system $(X_{\vartheta}, S)$ is intrinsically ergodic. Moreover, the unique measure of maximal entropy is the frequency measure corresponding to uniform probabilities.
\end{theorem}

The proof of \Cref{T intrinsic ergodicity} is presented in \Cref{SS intrinsic ergodicity proof}. We note that the subshifts considered in \Cref{T intrinsic ergodicity} do not satisfy the specification property of Bowen \cite{Bowen74} or the weaker specification property of Climenhaga and Thompson \cite{THOMPSON_3}.
Compare also \Cref{REM:beyond-specification} below.

\subsection{Gibbs properties of frequency measures}\label{4.7}

The proof of \Cref{T intrinsic ergodicity} follows a similar approach to the proof that the Parry measure is the unique measure of maximal entropy for irreducible shifts of finite type, due to Adler and Weiss \cite{adler-weiss}. An important feature of their proof is a Gibbs property, which states that for the measure of maximal entropy $\mu$, there exist constants $A,B > 0$ such that $A \me^{-\lvert u \rvert h} \leq \mu ([u]) \leq B \me^{-\lvert u \rvert h}$ for every legal word $u$, where $h$ denotes the topological entropy of the system. Such a Gibbs property does not hold for the subshifts considered in \Cref{T intrinsic ergodicity}. However, we can obtain a weaker Gibbs property for cylinder sets of exact inflation words. This is the content of \Cref{LEM: recog Gibbs property}, which utilises the following auxiliary results.

\begin{lemma}\label{LEM:mu-recursion}
Let $\vartheta_{\mathbf{P}} = (\vartheta, \mathbf{P})$ be a primitive random substitution with corresponding frequency measure $\mu_{\mathbf{P}}$. If for every $a_i \in \mathcal{A}$, the length $\lvert \vartheta (a_i) \rvert$ is well-defined, then for all $v \in \mathcal{L}_{\vartheta}$ and $w \in \vartheta(v)$,
	\begin{align*}
	\mu_{\mathbf{P}} ([w]) \geq  \frac{1}{\lambda} \mu_{\mathbf{P}} ([v]) \mathbb{P} [\vartheta_{\mathbf{P}} (v) = w].
	\end{align*} 
If in addition, $\vartheta_{\mathbf{P}}$ is recognisable and constant length, and $\lvert \vartheta(v) \rvert > 2 \kappa(\vartheta)$, then
	\begin{align*}
	\mu_{\mathbf{P}} ([w]) =  \frac{1}{\lambda}
	\sum_{u \in \mathcal L_{\vartheta}^{|v|}} \mu_{\mathbf{P}} ([u]) \mathbb{P} [\vartheta_{\mathbf{P}} (u) = w].
	\end{align*}
\end{lemma}

\begin{proof}
 Let $v \in \mathcal{L}_{\vartheta}$ and let $w \in \vartheta(v)$ be fixed. 
Let $n = |w|$ and $\mathcal{J}_{n}(v) = \{ u \in \mathcal{L}_{\vartheta}^{n} : u_{[1, \lvert v \rvert]} = v \}$. 
Since we assumed that the lengths of inflation words are well-defined, the relation in \Cref{LEM:key} simplifies to
\[
\mu_{\mathbf{P}} ([w])  = \frac{1}{\lambda} \sum_{u \in \mathcal{L}_{\vartheta}^{n}} \mu_{\mathbf{P}} ([u]) \sum_{j = 1}^{\lvert \vartheta(u_1) \rvert} \mathbb{P} [\vartheta_{\mathbf{P}} (u)_{[j, j + \lvert w \rvert - 1]} = w] 
\]
Using that $[v]$ is the union of all $[u]$ with $u \in \mathcal{J}_{n}(v)$ we thereby obtain
	\begin{align*}
	\mu_{\mathbf{P}} ([w])  & \geq \frac{1}{\lambda} \sum_{u \in \mathcal{J}_{n}(v)} \mu_{\mathbf{P}} ([u]) \mathbb{P} [\vartheta_{\mathbf{P}} (u)_{[1,\lvert w \rvert]} = w]
	= \frac{1}{\lambda} \sum_{u \in \mathcal{J}_{n}(v)} \mu_{\mathbf{P}} ([u]) \mathbb{P} [\vartheta_{\mathbf{P}} (v) = w] 
\\	& = \frac{1}{\lambda} \mu_{\mathbf{P}} ([v]) \mathbb{P} [\vartheta_{\mathbf{P}} (v) = w].
	\end{align*}
If $\vartheta_{\mathbf{P}}$ is recognisable and of constant length, and $\lvert \vartheta (v) \rvert > 2 \kappa (\vartheta)$, then there is a unique way to decompose $w$ into inflation words. 
However, there might still be several words $u \in \mathcal L_{\vartheta}$ with $|u| = |v|$ such that $w \in \vartheta(u)$.
\Cref{LEM:key} yields 
	\[
\mu_{\mathbf{P}} ([w]) =  \frac{1}{\lambda}
	\sum_{u \in \mathcal L_{\vartheta}^{|v|}} \mu_{\mathbf{P}} ([u]) \mathbb{P} [\vartheta_{\mathbf{P}} (u) = w].
	\qedhere
	\]
\end{proof}

\begin{lemma}\label{LEM:P-estimate}
Let $\vartheta_{\mathbf{P}} = (\vartheta, \mathbf{P})$ be a primitive random substitution satisfying the disjoint set condition. Assume that $\mathbb{P} [\vartheta_{\mathbf{P}} (a) = s] = 1/\# \vartheta (a)$ for all $a \in \mathcal{A}$ and $s \in \vartheta (a)$ and that at least one of the following conditions is satisfied:
	\begin{itemize}
	\item[(i)] $\vartheta_{\mathbf{P}}$ is of constant length $\ell$ and $\# \vartheta(a) = \# \vartheta(b)$ for all $a, b \in \mathcal{A}$;
	\item[(ii)] $\vartheta_{\mathbf{P}}$ is compatible and the second largest eigenvalue $\tau$ of the substitution matrix satisfies $\lvert \tau \rvert < 1$.
	\end{itemize}
Under these hypotheses, there exists a constant $c > 0$ such that $\mathbb{P}[\vartheta_{\mathbf{P}}^{m}(a) = w] \geq c \me^{- \lvert w \rvert h_{\text{top}}(X_{\vartheta})}$ for all $m \in \mathbb{N}$, $a \in \mathcal{A}$ and $w \in \vartheta^{m}(a)$. In particular, when $\vartheta_{\mathbf{P}}$ is of constant length, we have that $\mathbb{P} [\vartheta_{\mathbf{P}}^m (a) = w] = \me^{h_{\text{top}}(X_{\vartheta})} \me^{-|w| h_{\text{top}}(X_{\vartheta})}$.
\end{lemma}

\begin{proof}
As $\vartheta_{\mathbf{P}}$ satisfies the disjoint set condition, by induction, for $a \in \mathcal{A}$, $m \in \mathbb{N}$ and $w \in \vartheta^{m} (a)$,
	\begin{equation}\label{EQ: uniform distribution probs}
	\mathbb{P}[\vartheta_{\mathbf{P}}^{m} (a) = w] = \frac{1}{\# \vartheta^{m}(a)}.
	\end{equation} 
Let us first consider case (i). Since $\vartheta_{\mathbf{P}}$ satisfies the disjoint set condition, we have $\# \vartheta^{m}(a) = \# \vartheta^{m}(b)$ for all $a, b \in \mathcal{A}$. Hence, by \Cref{C IS DS topological entropy}, and since the right Perron--Frobenius eigenvector $\mathbf{R}$ of $\vartheta_{\mathbf{P}}$ is normalised so that $\lVert \mathbf{R} \rVert_{1} = 1$, we have
	\begin{align*}
	\log(\# \vartheta^{m}(a))
	= \sum_{b \in \mathcal{A}} R_b \log(\#\vartheta^{m}(b))
	= (\ell^{m} - 1)h_{\text{top}}(X_{\vartheta})  = \lvert \vartheta^{m}(a) \rvert h_{\text{top}}(X_{\vartheta}) - h_{\text{top}}(X_{\vartheta}).
	\end{align*}
Taking the exponential of both sides, we conclude from \eqref{EQ: uniform distribution probs} that $\mathbb{P} [\vartheta_{\mathbf{P}} (a) = w] = \me^{h_{\text{top}}(X_{\vartheta})} \me^{-|w| h_{\text{top}}(X_{\vartheta})}$.
Let us now consider case (ii).  Since the Perron--Frobenius eigenvalue $\lambda$ of $\vartheta_{\mathbf{P}}$ is simple, we can split the substitution matrix $M$ as $M = \lambda \mathbf{R} \mathbf{L}^{\top} + N$, where $\mathbf{R}$ and $\mathbf{L}$ are respectively the right and left Perron--Frobenius eigenvectors of $\vartheta_{\mathbf{P}}$ and where $N \mathbf{R} \mathbf{L}^{\top} = 0 = \mathbf{R} \mathbf{L}^{\top} N$. Since $\vartheta_{\mathbf{P}}$ satisfies the disjoint set condition, it follows by \cite[Lemma 10]{gohlke}, that $\mathbf{q}_{m}^{\top} = \mathbf{q}_{1}^{\top} \sum_{k=0}^{m-1} M^{k}$, for all $m \in \mathbb{N}$, and where $\mathbf{q}_{m}$ is as defined in \eqref{eq:refece_page_17}.  Hence,
	\begin{align*}
	\mathbf{q}_{m}^{\top}
	= \mathbf{q}_{1}^{\top} \sum_{k=0}^{m-1} M^{k}
	&= \mathbf{q}_1^{\top} \sum_{k = 0}^{m-1} \lambda^{k} \mathbf{R} \mathbf{L}^{\top} + \mathbf{q}_1^{\top} \sum_{k=0}^{m-1} N^{k}\\
	&= \frac{\lambda^{m} - 1}{\lambda - 1} \mathbf{q}_1^{\top} \mathbf{R} \mathbf{L}^{\top} + \mathbf{q}_1^{\top} \sum_{k=0}^{m-1} N^{k}
	 = (\lambda^{m} - 1) h_{\text{top}}(X_{\vartheta}) \mathbf{L}^{\top} + \mathbf{q}_1^{\top} \sum_{k=0}^{m-1} N^{k}.
	\end{align*}
By construction, $\tau$ is the dominant eigenvalue of $N$, and so there exists a $c > 0$ and $n \in \N$ such that $\lVert N^{k} \rVert_{\infty} < c k^n \lvert \tau \rvert^{k}$ for all $k \in \mathbb{N}$. Hence, there is $r \in \R$ with $|\tau| < r < 1$ such that $\lVert N^{k} \rVert_{\infty} < c r^k$. We therefore obtain
	\begin{align*}
	\log(\# \vartheta^{m}(a)) = q_{m,a} \!
	\leq (\lambda^{m} - 1)L_{a} h_{\text{top}}(X_{\vartheta}) + \! \lVert \mathbf{q}_1 \rVert_{\infty} \! \sum_{k = 0}^{m-1} \lVert N^{k} \rVert_{\infty}
	\leq (\lambda^{m} - 1)L_{a} h_{\text{top}}(X_{\vartheta}) + \! \frac{c}{1 - { r }} \lVert \mathbf{q}_1 \rVert_{\infty},
	\end{align*}
where $q_{m,a}$ is as defined in \eqref{eq:refece_page_17}. On the other hand, by \Cref{L letter frequencies}, we have that
	\begin{align*}
	\lvert \vartheta^{m}(a) \rvert \geq L_{a} \lambda^{m} - D \lvert \tau \rvert^{m} \geq L_{a} \lambda^{m} - D,
	\end{align*} 
for some $D>0$. Hence, there exists a constant $C >0$ such that $\log(\# \vartheta^{m}(a)) \leq \lvert \vartheta^{m}(a) \rvert  h + C$. Taking the exponential of both sides, we conclude from \eqref{EQ: uniform distribution probs} that $\mathbb{P} [\vartheta_{\mathbf{P}}^m (a) = w] \geq \me^{-|w| h} \me^{-C}$. Setting $c = e^{-C}$ completes the proof. If $\vartheta_{\mathbf{P}}$ is additionally assumed to be of constant length, then $\tau = 0$ since the eigenvalues of the matrix associated to a constant length substitution are integers. In this case, the matrix $M$ satisfies $M = \lambda \mathbf{R} \mathbf{L}^{\top}$, where $\mathbf{L} = (1, \ldots, 1)$ by the constant length property. Thus, it follows by the same arguments as above that $\log (\# \vartheta^{m} (a)) = (\lambda^m - 1) h_{\text{top}}(X_{\vartheta})$. Taking the exponential of both sides, it follows from \eqref{EQ: uniform distribution probs} that $\mathbb{P} [\vartheta_{\mathbf{P}}^m (a) = w] = \me^{h_{\text{top}}(X_{\vartheta})} \me^{-|w| h_{\text{top}}(X_{\vartheta})}$.
\end{proof}

\begin{lemma}\label{LEM:mu-bound}
If $\vartheta_{\mathbf{P}} = (\vartheta, \mathbf{P})$ satisfies either of the conditions of \Cref{LEM:P-estimate}, and if $\mu_{\mathbf{P}}$ denotes the corresponding frequency measure, then there exists a constant $c >0$ such that
	\begin{align*}
	\mu_{\mathbf{P}} ([w]) \geq \mu_{\mathbf{P}} ([v]) \frac{c^{\lvert v \rvert}}{\lvert w \rvert \me^{\lvert w \rvert h_{\text{top}}(X_{\vartheta})}}
	\end{align*} 
for all $v \in \mathcal{L}_{\vartheta}$, $m \in \mathbb{N}$ and $w \in \vartheta^{m} (v)$. If, in addition, $\vartheta_{\mathbf{P}}$ is constant length and recognisable and $\lvert v \rvert >  2 \kappa (\vartheta)$, then
	\begin{align*}
	\mu_{\mathbf{P}} ([w]) \leq \frac{\lvert v \rvert \me^{\lvert v \rvert h_{\text{top}}(X_{\vartheta})}}{\lvert w \rvert \me^{\lvert w \rvert h_{\text{top}}(X_{\vartheta})}} \text{.}
	\end{align*} 
\end{lemma}

\begin{proof}
Let $v \in \mathcal{L}_{\vartheta}$, $m \in \mathbb{N}$ and $w \in \vartheta^{m}(v)$ be fixed. Applying \Cref{LEM:mu-recursion} to $\vartheta_{\mathbf{P}}^{m}$ yields
	\begin{align}\label{EQN:mu bound}
	\mu_{\mathbf{P}} ([w]) \geq \frac{1}{\lambda^{m}} \mu_{\mathbf{P}} ([v]) \mathbb{P}[\vartheta_{\mathbf{P}}^{m} (v) = w].
	\end{align}
Since $\vartheta_{\mathbf{P}}$ is compatible or constant length, we can decompose $w$ into subwords $w = w^{(1)} \cdots w^{(\lvert v \rvert)}$ such that $w^{(j)} \in \vartheta^{m} (v_{j})$ for all $j \in \{ 1, \ldots, \lvert v \rvert \}$. Hence, it follows by \Cref{LEM:P-estimate} that there is a constant $c > 0$ such that
	\begin{align}\label{eq:tripple_star}
	\mathbb{P}[\vartheta_{\mathbf{P}}^{m} (v) = w]
	= \prod_{j = 1}^{\lvert v \rvert} \mathbb{P}[\vartheta_{\mathbf{P}}^{m} (v_j) = w^{(j)}]
	\geq \prod_{j=1}^{\lvert v \rvert} c \, \me^{-\lvert w^{(j)} \rvert h_{\text{top}}(X_{\vartheta})} = c^{\lvert v \rvert} \me^{-\lvert w \rvert h_{\text{top}}(X_{\vartheta})}.
	\end{align} 
By \Cref{L letter frequencies}, there is a universal constant $D > 0$ such that $\lambda^{m} \leq D \lvert \vartheta^{m}(a) \rvert$ for all $m \in \mathbb{N}$ and $a \in \mathcal{A}$. Combining this with \eqref{EQN:mu bound} and \eqref{eq:tripple_star} yields the required result.

Now, assume additionally that $\vartheta_{\mathbf{P}}$ is recognisable and of constant length $\ell$. Then by \Cref{LEM:P-estimate} we have that $\mathbb{P} [\vartheta_{\mathbf{P}}^{m} (u) = w] = \me^{\lvert u \rvert h_{\text{top}}(X_{\vartheta})} \me^{-\lvert w \rvert h_{\text{top}}(X_{\vartheta})}$ for every $u \in \mathcal{L}_{\vartheta}^{|v|}$ with $w \in \vartheta(u)$. Thus, the lower bound follows by identical arguments to the above, taking $c = e^{h_{\text{top}}(X_{\vartheta})}$. For the upper bound, observe that if $\lvert v \rvert > 2 \kappa (\vartheta)$, we also have $|\vartheta^m(v)| = \ell^m |v| > 2 \kappa (\vartheta^m)$, for all $m \in \mathbb{N}$ by \Cref{LEM:recognisability-recursion}. Hence, noting that $|u| = |v|$ and $\ell^{-m} = \lvert v \rvert/\lvert w \rvert$, we find by \Cref{LEM:mu-recursion},
\[
\mu_{\mathbf{P}} ([w]) =  \frac{1}{\ell^m}
	\sum_{u \in \mathcal L_{\vartheta}^{|v|}} \mu_{\mathbf{P}} ([u]) \mathbb{P} [\vartheta^m_{\mathbf{P}} (u) = w]
	\leq  \frac{\lvert v \rvert \me^{\lvert v \rvert h_{\text{top}}(X_{\vartheta})}}{ \lvert w \rvert \me^{\lvert w \rvert h_{\text{top}}(X_{\vartheta})}}.
	\qedhere
\]
\end{proof}

In the proof of \Cref{T intrinsic ergodicity} we only require the lower bound of \Cref{LEM:mu-bound}. However, the upper bound allows us to show that the subshifts we consider in \Cref{T intrinsic ergodicity} do not satisfy the Gibbs property, therefore do not satisfy the specification property of \cite{Bowen74}.
Instead, these subshifts satisfy the following Gibbs-like property.

\begin{lemma}\label{LEM: recog Gibbs property}
Let $\vartheta_{\mathbf{P}}$ be a random substitution satisfying the conditions of \Cref{T intrinsic ergodicity}, and let $\mu_{\mathbf{P}}$ denote the corresponding frequency measure. Then there exist constants $c_1, c_2 > 0$ such that for all $a \in \mathcal{A}$, $m \in \mathbb{N}$ and $w \in \vartheta^m (a)$,
\begin{equation*}
    \frac{c_1}{|w|} \me^{-|w| h_{\text{top}}(X_{\vartheta})} \leq \mu_{\mathbf{P}} ([w]) \leq \frac{c_2}{|w|} \me^{-|w| h_{\text{top}}(X_{\vartheta})} \text{.}
\end{equation*}
\end{lemma}
\begin{proof}
The lower bound follows immediately from \Cref{LEM:mu-bound}, taking $c_1 = \min_{a \in \mathcal{A}} c \mu_{\mathbf{P}} ([a])$ where $c$ is the constant given by \Cref{LEM:mu-bound}. For the upper bound, let $M$ be the least integer such that $\ell^{M} > 2 \kappa (\vartheta)$ and set $c_2 = \max_{u \in \mathcal{L}_{\vartheta}, \, |u| \leq \ell^{M}} \lvert u \rvert \me^{\lvert u \rvert h_{\text{top}}(X_{\vartheta})}$. Clearly $\mu_{\mathbf{P}} ([w]) \leq c_2 \me^{-\lvert w \rvert h_{\text{top}}(X_{\vartheta})} / \lvert w \rvert$ if $\lvert w \rvert \leq \ell^{M}$, since $\mu_{\mathbf{P}}$ is a probability measure. On the other hand, if $m > M$ and $w \in \vartheta^{m} (a)$ then it follows by \Cref{LEM:mu-bound} that there is a $v \in \vartheta^{M} (a)$ such that
\begin{equation*}
    \mu_{\mathbf{P}} ([w]) \leq \frac{\lvert v \rvert \me^{\lvert v \rvert h_{\text{top}}(X_{\vartheta})}}{\lvert w \rvert \me^{\lvert w \rvert h_{\text{top}}(X_{\vartheta})}} \leq \frac{c_2}{\lvert w \rvert} \me^{-\lvert w \rvert h_{\text{top}}(X_{\vartheta})} \text{.}
\end{equation*}
\end{proof}

\begin{remark}
\label{REM:beyond-specification}
The upper bound on $\mu_{\mathbf{P}}$ in \Cref{LEM: recog Gibbs property} is irreconcilable with the bound for the unique measure of maximal entropy on subshifts with a weak specification property established in \cite[Lemma~5.12]{THOMPSON_3}. For the subshifts $X_{\vartheta}$ with random substitutions as in \Cref{T intrinsic ergodicity}, $\mu_{\mathbf{P}}$ (with $\mathbf{P}$ the uniform distribution) is the unique measure of maximal entropy. Hence, each such $X_{\vartheta}$ does not satisfy the weak specification property in \cite{THOMPSON_3}. In particular, \Cref{T intrinsic ergodicity} establishes intrinsic ergodicity for subshifts beyond the more classical context of subshifts with (weak) specification.
\end{remark}

\subsection{Proof of \Cref{T intrinsic ergodicity}}\label{SS intrinsic ergodicity proof}

We now present the proof of \Cref{T intrinsic ergodicity}. In addition to the Gibbs property proved in the previous section, we also utilise the following result, which is proved in \cite{einsiedler-lindenstrauss-ward}.

\begin{lemma}[\cite{einsiedler-lindenstrauss-ward}, Lemma 8.8]\label{LEM:approx}
Let $(X,d)$ be a compact metric space and let $\varrho$ be a Borel probability measure on $X$. If $B \subset X$ is measurable and $(\xi_{n})_{n \in \mathbb{N}}$ is a sequence of finite measurable partitions of $X$ for which $\lim_{n \to \infty} \max_{P \in \xi_n} \operatorname{diam}(P) = 0$, then there exists a sequence of  sets $(A_{n})_{n \in \mathbb{N}}$ with $A_{n} \in \sigma(\xi_{n})$ and $\lim_{n \to \infty} \varrho(A_{n} \triangle B) = 0$. Here, $\sigma(\xi_{n})$ denotes the sigma algebra generated by the partition $\xi_{n}$.
\end{lemma}

\begin{proof}[Proof of \Cref{T intrinsic ergodicity}]
Let $\mu$ denote the frequency measure of maximal entropy given by \Cref{T IS DS MME} or \Cref{LEM:noncompatible top entropy}, and let $m \in \mathbb{N}$. For each $k \in \{ 0, \ldots, \ell^{m} -1 \}$, let $X_{m,k}$ denote the subset of $X_{\vartheta}$ defined by $X_{m,k} = S^{k}(\vartheta^{m}(X_{\vartheta}))$.  It follows by recognisability that these subsets are pairwise disjoint for different choices of $k$. Note, by \Cref{L compact to compact} the subsets $X_{m, k}$ are closed, and since by the constant length property
	\begin{align*}
	S^{\ell^{m}} (X_{m,k})
	= S^{\ell^{m}} (S^{k} (\vartheta^{m}(X_{\vartheta})))
	= S^{k} (\vartheta^{m}(S X_{\vartheta}))
	= S^{k}(\vartheta^{m}(X_{\vartheta})) = X_{m,k},
	\end{align*} 
$X_{m, k}$ is $S^{\ell^{m}}$-invariant.  In other words, $X_{m, k}$ is a subshift under $S^{\ell^{m}}$. Since every $x \in X_\vartheta$ can be split into level $m$ inflation words, we have
	\begin{align*}
	X_\vartheta = \bigsqcup_{k = 0}^{\ell^{m} - 1} X_{m,k},
	\end{align*} 
where the union is disjoint due to recognisability.  \Cref{LEM:recognisability-recursion} implies that $r = \ceil{\kappa(\vartheta)/(\ell -1)} + 1$ satisfies
	\begin{align*}
	\ell^{m} r > \frac{\ell^{m} -1}{\ell -1} \kappa(\vartheta) + \ell^{m} \geq \kappa(\vartheta^{m}) + \ell^{m}.
	\end{align*} 
By the constant length property, this ensures that every word of length at least $2 r \ell^{m}$ has a unique decomposition into inflation words. This together with \Cref{LEM:local-recog} implies,  for all $u \in \mathcal{L}_{\vartheta}^{2r}$ and $w \in \vartheta^{m}(u)$, that $\lvert w \rvert = 2r \ell^{m}$ and $S^{r \ell^{m}}([w]) \subset \vartheta^{m}(X_{\vartheta})$. Let us consider the following partition of $X_{m,k}$:
	\begin{align*}
	\xi_{m,k} = S^{r \ell^{m}} \left( \left\{ S^{k}([w]) : w \in \vartheta^{m}(u) \; \text{and} \; u \in \mathcal{L}_{\vartheta}^{2r}\right\} \right).
	\end{align*} 
This in turn yields a partition of $X_{\vartheta}$, namely
	\begin{align*}
	\xi_{m} = \bigcup_{k = 0}^{\ell^{m} - 1} \xi_{m,k}.
	\end{align*} 
By way of a contradiction, assume that $\nu \neq \mu$ is another ergodic measure of maximal entropy. Since distinct ergodic measures are mutually singular, there exists an $S$-invariant set $B$ with $\mu (B) = 0$ and $\nu (B) = 1$. Note, the diameter of the atoms of $\xi_m$ tends uniformly to zero as $m$ tends to infinity, and so $(\xi_m)_{m \in \mathbb{N}}$ meets the requirements of \Cref{LEM:approx}. Applying it to the measure $\varrho' = (\mu + \nu)/2$ we obtain that, given $\varepsilon > 0$, there exists $m \in \mathbb{N}$ and $A_m \in \sigma(\xi_m)$ such that
	\begin{align}\label{EQ:symm-difference}
	(\mu + \nu) (A_m \triangle B) < \varepsilon.
	\end{align}
For $k \in \{ 0, \ldots, \ell^{m} - 1\}$, let $A_{m,k} = A_m \cap X_{m,k}$ and $B_{m,k} = B \cap X_{m,k}$, and define the conditional probability measures $\mu_{m,k}$ and $\nu_{m,k}$ by
	\begin{align*}
	\mu_{m,k} = \frac{1}{\mu(X_{m,k})} \, \mu \vert_{X_{m,k}}
	\quad \text{and} \quad
	\nu_{m,k} = \frac{1}{\nu(X_{m,k})} \, \nu \vert_{X_{m,k}}.
	\end{align*} 
For all $j \in \{ 0, \ldots, \ell^{m} - 1 \}$, we have $S^{k-j} (X_{m,j}) = X_{m,k}$, and since $\mu$ and $\nu$ are $S$-invariant and since the sets $X_{m,k}$ are disjoint, it follows that
	\begin{align*}
	\mu(X_{m,k}) =  \mu(X_{m,j}) = \frac{1}{\ell^{m}}
	\quad \text{and} \quad
	\nu(B \cap X_{m,k}) =  \nu(B \cap X_{m,j}) = \frac{1}{\ell^{m}}.
	\end{align*} 
Consequently, $\nu_{m,k}(B_{m,k}) = \ell^{m} \, \nu(B \cap X_{m,k}) = 1$. On the other hand, $\mu_{m,k} (B_{m,k}) = \ell^{m} \mu (B \cap X_{m,k}) =  0$. Since $\{ X_{m,k} : k \in \{ 0, \ldots, \ell^{m} -1 \} \}$ forms a partition of $X_{\vartheta}$, we can rewrite \eqref{EQ:symm-difference} as
	\begin{align*}
	\sum_{k = 0}^{\ell^{m} -1} (\mu_{m,k} + \nu_{m,k})(A_{m,k} \triangle B_{m,k}) 
	= \ell^{m} \sum_{k=0}^{\ell^{m}-1} (\mu + \nu)((A_m \triangle B) \cap X_{m,k})
	= \ell^{m} (\mu + \nu) (A_m \triangle B) < \ell^{m} \varepsilon.
	\end{align*}
Hence, there exists a $k'$ such that
	\begin{align}\label{EQ:symm-difference-2}
	(\mu_{m,k'} + \nu_{m,k'})(A_{m,k'} \triangle B_{m,k'}) < \varepsilon.
	\end{align}
Here we observe that $A_{m,k'} \in \sigma(\xi_{m,k'})$, and recall, if $\lvert v \rvert \geq 2 \ell^{m} r$, then the word $v$ has a unique inflation word decomposition under $\vartheta^{m}$. Therefore, there exists a unique $j \in \{ 0, \ldots, \ell^{m} - 1\}$ such that $[v] \subset X_{m,j}$.

Note that the system $(X_{m,j}, S^{\ell^m})$ equipped with the measure $\nu_{m,j}$ is an induced subshift obtained from $(X_{\vartheta},S)$ equipped with the measure $\nu$ by inducing on $X_{m,j}$. Hence, by Abramov's formula \cite{Abramov_59},
	\begin{align*}
	 h(S, \nu) = \frac{1}{\ell^m} h(S^{\ell^{m}}\!, \nu_{m,j}).
	\end{align*} 
We now proceed by similar arguments to Adler and Weiss' \cite{adler-weiss} proof that Markov shifts are intrinsically ergodic, applied to the system $(X_{m,k^{'}}, S^{\ell^{m}})$ and the $S^{\ell^{m}}$-invariant measures $\mu_{m,k^{'}}$ and $\nu_{m,k^{'}}$. For ease of notation, in the following we write $k = k^{'}$ and $T = S^{\ell^{m}}$. Note that 
	\begin{align*}
	\alpha_{m,k} = \{ S^{k}([w]) : w \in \vartheta^{m}(a), a \in \mathcal{A} \}
	\end{align*} 
forms a generating partition of $X_{m,k}$, and by the fact that $\vartheta_{\mathbf{P}}$ is of constant length and recognisable,
	\begin{align*}
	\xi_{m,k} = \bigvee_{j = -r}^{r-1} T^{-j}(\alpha_{m, k}).
	\end{align*} 
Let $\eta_{m} = \{ A_{m,k}, X_{m,k} \setminus A_{m,k} \}$ and for a set $A \subseteq X_{m,k}$ denote by $t_{m}(A)$ the number of atoms in $\xi_{m,k}$ that intersect $A$. By definition, and using \eqref{eq:refence_page_20_i},
we have
	\begin{align*}
	2 r \ell^{m} h(S, \nu)
	&= 2 r h(S^{\ell^{m}}, \nu_{m,k})
	\leq H_{\nu_{m,k}}(\xi_{m,k})\\
	&\leq H_{\nu_{m,k}}(\eta_m) + H_{\nu_{m,k}}(\xi_{m,k} \vert \eta_m)\\
	& \leq \log(2) + \nu_{m,k}(A_{m,k}) \log(t_m(A_{m,k})) + \nu_{m,k}(X_{m,k} \setminus A_{m,k}) \log(t_m(X_{m,k} \setminus A_{m,k})).
	\end{align*}
Let $S^{r \ell^{m} + k}[w] \in \xi_{m,k}$, with $w \in \vartheta^{m}(v)$ for some $v \in \mathcal{L}_{\vartheta}^{2r}$. By \Cref{LEM:mu-bound}, we have that
	\begin{align*}
	\mu_{m,k} (S^{r \ell^{m} + k} ([w]))= \ell^{m} \mu([w]) \geq \mu([v]) \frac{c^{2r}}{2 r \me^{ 2 \ell^{m} r h_{\text{top}}(X_{\vartheta})}}
	\geq C  \me^{- 2 r \ell^{m} h_{\text{top}}(X_{\vartheta})},
	\end{align*} 
taking $C = c^{2r} (\min_{v \in \mathcal{L}_{\vartheta}^{2r}} \mu ([v])) / 2r $. We have that $C > 0$ since $\mu([v]) > 0$ for all $v \in \mathcal{L}_{\vartheta}^{r}$. Hence,
	\begin{align*}
	t_m(A_{m,k}) \leq \frac{1}{C} \, \mu(A_{m,k}) \me^{2 \ell^{m} r h_{\text{top}}(X_{\vartheta})}
	\quad \text{and} \quad 
	t_m(X_{m,k} \setminus A_{m,k}) \leq \frac{1}{C} \,\mu(X_{m,k} \setminus A_{m,k}) \me^{2 \ell^{m} r h_{\text{top}}(X_{\vartheta})}.
	\end{align*} 
This yields $0 \leq \log(2) - \log(C) + \nu_{m,k}(A_{m,k}) \log(\mu_{m,k}(A_{m,k}))$. By \eqref{EQ:symm-difference-2}, we have that $\mu_{m,k}(A_{m,k}) < \varepsilon $ and $\nu_{m,k}(A_{m,k}) > 1 - \varepsilon$. This implies the following contradiction:
	\[
	0 \leq \lim_{\varepsilon \to 0} \left(\log(2) - \log(C) + (1 - \varepsilon) \log(\varepsilon) \right) = - \infty.
	\qedhere
	\]
\end{proof}

From \Cref{LEM:mu-bound}, we have used only the lower bound in the proof of \Cref{T intrinsic ergodicity}. Since this inequality holds under less restrictive conditions, it seems natural to inquire whether \Cref{T intrinsic ergodicity} can be sharpened accordingly by replacing the constant length assumption with a weaker condition. However, a closer inspection reveals that the last part of the proof relies on the detailed control that the constant length assumption provides. A definite answer therefore remains as an open problem.

\section{Examples and open questions}\label{S examples}

In this section we present examples of random substitution subshifts that exhibit various properties. We first present several examples that illustrate the main results of this paper and their applications to two prototypical examples of random substitutions, the random period doubling (\Cref{EX: RPD example section}) and random Fibonacci (\Cref{EX: R Fib}) substitutions. We then consider some familiar examples of subshifts which can be obtained as subshifts of primitive random substitutions, including the golden mean shift (\Cref{EXA: golden mean shift}) and the Dyck shift (\Cref{EXA: Dyck shift}). A summary of the key properties of each of the examples is presented in the table below.

\begin{center}
\begin{tabular}{c|c|c|c|c|c|c|c}
 & 5.1 & 5.2 & 5.3 & 5.4 & 5.5 & 5.6 & 5.7 \\
 \hline 
Unique r. paths & \cmark & \cmark & \cmark & \cmark & \cmark & \cmark & \xmark \\
Compatible & \xmark & \cmark & \xmark & \cmark & \xmark & \cmark & \xmark \\
Constant length & \cmark & \cmark & \cmark & \xmark & \xmark & \cmark & \xmark \\
(ISC)/(DSC) & (DSC) & \xmark & (DSC)  & (DSC) & (DSC) & (ISC) & \xmark \\
Recognisable & \cmark & \xmark & \xmark & \xmark & \xmark & \xmark & \xmark \\
Frequency MME & \cmark & \cmark & \xmark & \xmark & \cmark & \cmark & ? \\
Intrinsically ergodic & \cmark & ? & ? & ? & \cmark & \cmark & \xmark \\
\end{tabular}
\end{center}

By the existence of a frequency measure of maximal entropy, we mean that there exists a choice of probabilities on the given set-valued substitution that gives rise to a frequency measure of maximal entropy. In particular, when we say there does not exist such a frequency measure of maximal entropy, we do not rule out the possibility that there exists another random substitution that gives rise to the same subshift for which the corresponding frequency measure is a measure of maximal entropy.

We first give an example of a random substitution which satisfies the conditions of \Cref{T intrinsic ergodicity}, thus gives rise to an intrinsically ergodic subshift.

\begin{example}\label{EX: intrinsically ergodic example}
Let $\vartheta$ be the random substitution defined by
	\begin{align*}
	\vartheta \colon
		\begin{cases}
		a \mapsto
			\begin{cases}
			aaa & \text{with probability } 1/2,\\
			abb & \text{with probability } 1/2,
			\end{cases}\\[1.25em]
		b \mapsto
			\begin{cases}
			bba & \text{with probability } 1/2,\\
			aba & \text{with probability } 1/2,
			\end{cases}
		\end{cases}
	\end{align*}
with associated subshift $X_{\vartheta}$ and corresponding frequency measure $\mu$. One can verify that $\vartheta$ is recognisable and satisfies the conditions of \Cref{T intrinsic ergodicity} (specifically (i)). Hence, $\mu$ is the unique measure of maximal entropy for the system $(X_{\vartheta},S)$. By \Cref{THM:main-urp}, we have 
\begin{equation*}
    h (\mu) = h_{\text{top}} (X_{\vartheta}) = \frac{1}{2}\log(2) \text{.}
\end{equation*}
\end{example}

\begin{example}[Random period doubling]\label{EX: RPD example section}
Let $p \in (0,1)$, let $\vartheta_{p}$ be the random substitution defined by
	\begin{align*}
	\vartheta_{p} \colon
		\begin{cases}
		a \mapsto 
			\begin{cases}
			ab & \text{with probability} \; p,\\
			ba & \text{with probability} \; 1-p,
			\end{cases}\\[1.25em]
		b \mapsto aa,
		\end{cases}
	\end{align*}
and let $\mu_{p}$ denote the corresponding frequency measure. We have that $\vartheta_{p}$ is compatible and satisfies the disjoint set condition, so it follows by \Cref{THM:main-urp} that
	\begin{align*}
	h (\mu_{p}) = -\frac{2}{3} (p \log(p) + (1-p) \log(1-p)).
	\end{align*}
Moreover, by \Cref{T IS DS MME}, we have that $\mu_{1/2}$ is a measure of maximal entropy for the system $(X_{\vartheta},S)$. It is known that $\vartheta_p$ is not recognisable; therefore, we are unable to apply \Cref{T intrinsic ergodicity}, so it remains open as to whether or not this is the unique measure of maximal entropy.
\end{example}

For each of the previous two examples, the frequency measure corresponding to uniform probabilities was a measure of maximal entropy. However, this is not the case for all primitive random substitutions satisfying the disjoint set condition, as is demonstrated by the following example. Here, the frequency measure of greatest entropy occurs at a non-uniform choice of probabilities, and this frequency measure is not a measure of maximal entropy.

\begin{example}\label{EX: R dependent on probs}
Let $p \in (0,1)$ and let $\vartheta_{\mathbf{P}}$ be the random substitution defined by
\begin{equation*}
    \vartheta_{p} \colon
    \begin{cases}
    a \mapsto
    \begin{cases}
    aa &\text{with probability $p$,}\\
    ab &\text{with probability $1-p$,}
    \end{cases}\\
    b \mapsto ba,
    \end{cases}
\end{equation*}
with corresponding frequency measure $\mu_p$ and subshift $X_{\vartheta}$. Since $\vartheta_{p}$ is constant length and satisfies the disjoint set condition, it follows by \Cref{THM:main-urp} that
\begin{equation*}
    h (\mu_p) = - \frac{1}{2-p} (p \log p + (1-p) \log (1-p)) \text{.}
\end{equation*}
The value of $p$ that maximises the above expression is $p = \tau^{-1}$, where $\tau$ is the golden ratio, for which the corresponding entropy is
\begin{equation*}
    h (\mu_{\tau^{-1}}) = \log \tau \approx 0.481212 \text{.}
\end{equation*}
On the other hand, one can compute that the topological entropy of the system $(X_{\vartheta}, S)$ is
\begin{equation*}
h_{\text{top}} (X_{\vartheta}, S) = \sum_{n=1}^{\infty} \frac{1}{2^n} \log n \approx 0.507834 \text{,}
\end{equation*}
so $\mu_{\tau^{-1}}$ is not a measure of maximal entropy. Thus, the conclusion of \Cref{T IS DS MME} does not hold without compatibility, even for constant length random substitutions. We note that the topological entropy equals $\log \sigma$, where $\sigma$ is Somos's quadratic recurrence constant \cite[p. 446]{finch}. It is an open question as to whether $\sigma$ is algebraic or transcendental. 
\end{example}
\begin{figure}[ht]\label{Entropy graph}
\centering
\includegraphics[scale = 0.5]{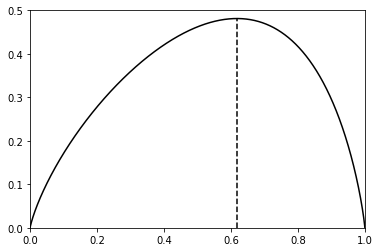}
\caption{Plot of $h (\mu_p)$ for $p \in (0,1)$.}
\end{figure}

The previous examples all satisfy the disjoint set condition, so we could obtain a closed form expression for the entropy via \Cref{THM:main-urp}. This is not the case for our next example, the random Fibonacci substitution, which is compatible but does not satisfy either the disjoint or identical set condition.

\begin{example}[Random Fibonacci]\label{EX: R Fib}
Let $\vartheta_{\text{RF}}$ denote the random substitution defined by
	\begin{align*}
	\vartheta_{\text{RF}} \colon
		\begin{cases}
		a \mapsto
			\begin{cases}
			ab & \text{with probability} \; 1/2,\\
			ba & \text{with probability} \; 1/2,
			\end{cases}\\[1.25em]
		b \mapsto a,
		\end{cases}
	\end{align*}
and let $\mu_{\text{RF}}$ denote the corresponding frequency measure. Since $\vartheta_{\text{RF}}$ satisfies neither the identical set condition nor the disjoint set condition, \Cref{THM:main-urp} does not yield a closed form formula for the measure theoretic entropy of $\mu_{\text{RF}}$. However, we may use \Cref{THM:main-urp} to obtain a sequence of bounds on $h (\mu_{\text{RF}})$.  Indeed, we have $\lambda^{-k} \mathbf{H}_{k}^{\top} \mathbf{R} \leq h (\mu_{\text{RF}}) \leq (\lambda^{k} - 1)^{-1} \mathbf{H}_{k}^{\top} \mathbf{R}$ for all $k \in \mathbb{N}$, and a computer-assisted calculation of $\mathbf{H}_{6}^{\top} \mathbf{R}$ yields
	\begin{align*}
	0.3908 < \frac{1}{\lambda^{6}} \mathbf{H}_{6}^{\top} \mathbf{R} \leq h (\mu_{\text{RF}}) \leq \frac{1}{\lambda^{6} - 1} \mathbf{H}_{6}^{\top} \mathbf{R} < 0.4140,
	\end{align*}
noting that $\lambda = \tau$, where $\tau$ is the golden ratio. It was shown in \cite{godreche-luck,nilsson2} that
	\begin{align*}
	h_{\text{top}}(X_{\vartheta_{\text{RF}}}) = \sum_{m=2}^{\infty} \frac{\log(m)}{\tau^{m+2}} \approx 0.444399,
	\end{align*}
so $\mu_{\text{RF}}$ is \emph{not} a measure of maximal entropy. By taking higher powers, we obtain frequency measures of greater entropy. If we consider the square of $\vartheta_{\text{RF}}$ with uniform probabilities, namely
	\begin{align*}
	\vartheta_{\text{RF}, 2} \colon
		\begin{cases}
		a \mapsto
			\begin{cases}
			baa & \text{with probability} \; 1/3,\\
			aba & \text{with probability} \; 1/3,\\
			aab & \text{with probability} \; 1/3,
			\end{cases}\\[2em]
		b \mapsto
			\,\begin{cases}
			ab\hphantom{a} & \text{with probability} \; 1/2,\\
			ba & \text{with probability} \; 1/2,
			\end{cases}
		\end{cases}
	\end{align*}
and let $\mu_{\text{RF}, 2}$ be the corresponding frequency measure, then by \Cref{THM:main-urp} and a computer-aided calculation of $\mathbf{H}_{3}^{\top} \mathbf{R}$ we obtain
	\begin{align*}
	0.4177 < \frac{1}{\lambda^{6}} \mathbf{H}_{3}^{\top} \mathbf{R} \leq h (\mu_{\text{RF},2}) \leq \frac{1}{\lambda^{6}-1} \mathbf{H}_{3}^{\top} \mathbf{R} < 0.4424.
	\end{align*}
Here, $\lambda^{2}$ is the Perron--Frobenius eigenvalue of $\vartheta_{\text{RF}, 2}$. Hence, $h (\mu_{\text{RF}}) < h (\mu_{\text{RF}, 2}) < h_{\text{top}}(X_{\vartheta_{\text{RF}}})$, so $\mu_{\text{RF}, 2} $ is still not a measure of maximal entropy, but has strictly greater entropy than $\mu_{\text{RF}}$. \Cref{T weak limit MME} gives that a measure of maximal entropy can be obtained as a weak limit of frequency measures. In particular, if $(\mu_{\text{RF}, n})_{n \in \mathbb{N}}$ is the sequence of frequency measures corresponding to the $n$-th power of $\vartheta_{\text{RF}}$ with uniform probabilities, then there exists a subsequence $(\mu_{\text{RF}, n_{k}})_{k \in \mathbb{N}}$ that converges weakly to a measure of maximal entropy. 
As to whether the system $(X_{\vartheta_{\text{RF}}}, S)$ is intrinsically ergodic, this remains open.
\end{example}

We now consider applications of our results to other common subshifts in symbolic dynamics. It was shown in \cite{gohlke-rust-spindeler} that every topologically transitive shift of finite type can be obtained as the subshift of a primitive random substitution. For the golden mean shift, it is possible to obtain the Parry measure as a weak limit of frequency measures corresponding to primitive random substitutions.

\begin{example}[The golden mean shift]\label{EXA: golden mean shift}
The golden mean shift is the shift of finite type over the alphabet $\{ a, b \}$ defined by the forbidden word set $\mathcal{F} = \{ bb \}$. The subshift $X$ can be obtained as the subshift of the random substitution
\begin{equation*}
    \vartheta \colon
    \begin{cases}
    a \mapsto
    \begin{cases}
    aa & \text{with probability $\tau^{-1}$,}\\
    aba & \text{with probability $\tau^{-2}$,}\\
    \end{cases}\\
    b \mapsto
    b.
    \end{cases}
\end{equation*}
However, this random substitution is not primitive, so we cannot directly apply our results. To circumvent this issue, let $\varepsilon \in (0,1)$ and let $\vartheta_{\varepsilon}$ be the random substitution defined by 
\begin{equation*}
    \vartheta_{\varepsilon} \colon
    \begin{cases}
    a \mapsto
    \begin{cases}
    aa & \text{with probability $\tau^{-1}$,}\\
    aba & \text{with probability $\tau^{-2}$,}\\
    \end{cases}\\
    b \mapsto
    \begin{cases}
    b & \text{with probability $1-\varepsilon$,}\\
    abb & \text{with probability $\varepsilon$,}
    \end{cases}
    \end{cases}
\end{equation*}
and let $\mu_{\varepsilon}$ denote the corresponding frequency measure. For all $\varepsilon \in (0,1)$, $\vartheta_{\varepsilon}$ is a primitive random substitution with unique realisation paths satisfying the disjoint set condition. Let $\mu$ be the weak limit of $\mu_{\varepsilon}$ as $\varepsilon \rightarrow 0$. By compactness, $\mu$ is a shift-invariant probability measure. Also note that $X$ is the support of $\mu$. One can show that $R_{a, \varepsilon} / (\lambda_{\varepsilon} - 1) \rightarrow \tau^2 / (\tau^2 +1)$ as $\varepsilon \rightarrow 0$, where $\lambda_{\varepsilon}$ and $R_{a, \varepsilon}$ are the Perron--Frobenius eigenvalue and the entry of the right Perron--Frobenius eigenvector corresponding to the letter $a$, respectively. Thus, it follows by the upper semi-continuity of entropy and \Cref{THM:main-urp} that
\begin{equation*}
\begin{split}
    h (\mu) &\geq \limsup_{\varepsilon \rightarrow 0} h (\mu_{\varepsilon}) = \limsup_{\varepsilon \rightarrow 0} \frac{1}{\lambda_{\varepsilon} - 1} \mathbf{H}_{1}^{\top} \mathbf{R} \geq \limsup_{\varepsilon \rightarrow 0} \frac{-1}{\lambda_{\varepsilon} - 1} R_{a,\varepsilon} (\tau^{-2} \log \tau^{-2} + \tau^{-1} \log \tau^{-1})\\
    &= \frac{\tau^2}{\tau^2+1} (2 \tau^{-2} + \tau^{-1}) \log \tau = \log \tau \text{,}
\end{split}
\end{equation*}
where in the last equality we have used the characteristic equation $\tau^2 = \tau + 1$. Since $h_{\text{top}} (X,S) = \log \tau$ and the Parry measure is the unique measure of maximal entropy \cite{adler-weiss, Parry_64}, we conclude that $\mu$ must be the Parry measure.
\end{example}

We note that the algorithm in \cite{gohlke-rust-spindeler} yields a primitive random substitution that gives rise to the golden mean shift. However, a closer inspection reveals that if the corresponding frequency measure is the Parry measure then we require two of the realisations to occur with probability zero and the resulting random substitution is the random substitution $\vartheta$ defined in \Cref{EXA: golden mean shift}, which is not primitive. As to whether there exists a primitive random substitution for which the Parry measure is the corresponding frequency measure remains open. Our next example is a sofic shift for the which the unique measure of maximal entropy can be obtained as a frequency measure of a primitive random substitution.

\begin{example}[A sofic shift]\label{EX: sofic example}
Let $p \in (0,1)$, let $\vartheta_{p}$ be the random substitution defined by
	\begin{align*}
	\vartheta_{p} \colon
	a, b \mapsto 
		\begin{cases}
		ab & \text{with probability} \; p,\\
		ba & \text{with probability} \; 1-p,
		\end{cases}
	\end{align*}
and let $\mu_{p}$ denote the corresponding frequency measure. In \cite[Proposition 6.7]{gohlke-spindeler}, the measure theoretic entropy of $\mu_p$ was calculated directly and shown to be
	\begin{align*}
	h (\mu_{p}) = - \frac{1}{2} (p \log(p) + (1-p) \log(1-p)).
	\end{align*}
Since $\vartheta_{p}$ satisfies the identical set condition and has identical production probabilities, \Cref{THM:main-urp} gives an alternative method of obtaining this formula. Moreover, by \Cref{T IS DS MME}, for $p = 1/2$, the measure $\mu_p$ is a measure of maximal entropy. Notice that $\vartheta_p$ is of constant length, but not recognisable since it does not satisfy the disjoint set condition. Hence, \Cref{T intrinsic ergodicity} may not be applied. However, it was shown in \cite[Corollary 6.8]{gohlke-spindeler} that the subshift associated to $\vartheta_{p}$ is a sofic shift, thus intrinsically ergodic. Hence, $\mu_{p}$ with $p = 1/2$ is the unique measure of maximal entropy for the system $(X_{\vartheta},S)$.
\end{example}

We finally present an example of a random substitution subshift which has multiple measures of maximal entropy. This is the \emph{Dyck shift}, which was shown in \cite{krieger} to support two distinct ergodic measures of maximal entropy.

\begin{example}[The Dyck shift]\label{EXA: Dyck shift}
For $i \in \{ 1,2,3,4 \}$, let $\mathbf{p}_i = (p_{i,1},p_{i,2},p_{i,3})$ be a probability vector and let $\mathbf{P} = \{ \mathbf{p}_1, \mathbf{p}_2, \mathbf{p}_3, \mathbf{p}_4 \}$. Define the random substitution $\vartheta_{\mathbf{P}}$ over the alphabet $\mathcal{A} = \{ (, \, ), \, [, \, ] \}$ by
\begin{align*}
    \vartheta_{\mathbf{P}} \colon
    \begin{cases} \;
   \begin{aligned}
    ( &&\mapsto&&
    \begin{cases}
    \; (  &\text{ with probability $p_{1,1}$,}\\
    \; ( ( ) &\text{ with probability $p_{1,2}$,}\\
    \; ( [ ] &\text{ with probability $p_{1,3}$,}\\
    \end{cases}
    &\quad& ) &&\mapsto&&
    \begin{cases}
    \; ) &\text{ with probability $p_{2,1}$,}\\
    \; ( ) ) &\text{ with probability $p_{2,2}$,}\\
    \; [ ] ) &\text{ with probability $p_{2,3}$,}\\
    \end{cases}\\
    [ &&\mapsto&&
    \begin{cases}
    \; [ &\text{ with probability $p_{3,1}$,}\\
    \; [ ( ) &\text{ with probability $p_{3,2}$,}\\
    \; [ [ ] &\text{ with probability $p_{3,3}$,}\\
    \end{cases}
    &\quad& ] &&\mapsto&&
    \begin{cases}
    \; ] &\text{ with probability $p_{4,1}$,}\\
    \; ( ) ] &\text{ with probability $p_{4,2}$,}\\
    \; [ ] ] &\text{ with probability $p_{4,3}$.}\\
    \end{cases}
\end{aligned}
    \end{cases}
\end{align*}
The corresponding subshift is the Dyck shift, which supports two distinct measures of maximal entropy \cite{gohlke-spindeler}. The random substitution $\vartheta_{\mathbf{P}}$ does not have unique realisation paths since, for example, the word $(())$ can be obtained as two different realisations of $()$ under $\vartheta_{\mathbf{P}}$. Consequently, it is difficult to verify whether or not either or both of the ergodic measures of maximal entropy can be obtained as frequency measures.
\end{example}

This final example motivates the following open question.

\begin{question}
Under what conditions does a primitive random substitution give rise to an intrinsically ergodic subshift?
\end{question}
We have presented three examples of random substitutions which give rise to intrinsically ergodic subshifts. In general it appears to be difficult to deduce whether a random substitution subshift is intrinsically ergodic. The absence of a Gibbs property and specification provide obstacles to adapting many of the conventional methods for checking whether a subshift is intrinsically ergodic. Further, there does not appear to be an easy way of extending the proof of \Cref{T intrinsic ergodicity} to the case where the substitution is not recognisable or constant length. As such, we leave a definitive answer to future work.

\section*{Acknowledgements}

It is a pleasure to thank R.\,Leek and J.\,Mitchell for many useful and helpful discussions. We are also very grateful for detailed feedback from an anonymous referee. The first and third authors were supported by the German Research Foundation (DFG) via the Collaborative Research Centre (CRC 1283) and by the Research Centre of Mathematical Modelling (RCM2) of Universit\"at Bielefeld.  The second author would like to thank the School of Mathematics at University of Birmingham and EPSRC DTP for their support.  Finally, the last author would like to thank HIM for supporting a research stay during the program \textsl{Dynamics:\ Topology and Numbers} where part of this work was completed.

\bibliographystyle{abbrv}
\bibliography{ref}

\end{document}